\documentclass[11pt]{article}

\usepackage[pagewise]{lineno}

\usepackage[english]{babel}
\usepackage{indentfirst}
\usepackage{latexsym}
\usepackage[colorlinks=true, bookmarksnumbered=true, bookmarksopen=true,
bookmarksopenlevel=3, pdfstartview=FitH, linkcolor=cyan, pdfmenubar=true,
pdftoolbar=true, bookmarks=true,citecolor=cyan, urlcolor=magenta,
filecolor=cyan,menucolor=black,plainpages=false,pdfpagelabels]{hyperref}
\usepackage[lmargin=2cm,tmargin=2cm,bmargin=2cm,rmargin=2cm]{geometry}
\usepackage{amsbsy, amsmath,amsfonts,amssymb,amsthm}
\usepackage{times}
\usepackage{graphicx}
\usepackage{amsmath}
\usepackage{amsfonts}
\usepackage{amssymb}
\usepackage[english]{babel}
\usepackage[latin1]{inputenc}
\usepackage[T1]{fontenc}
\usepackage{amsmath,amssymb,graphics,amsthm}
\usepackage[usenames]{color}

\numberwithin{equation}{section}
\newtheorem{prop}{Proposition}[section]
\newtheorem{defi}[prop]{Definition}
\newtheorem{teo}[prop]{Theorem}
\newtheorem{coro}[prop]{Corollary}
\newtheorem{obs}[prop]{Remark}
\newtheorem{lema}[prop]{Lemma}

\newtheorem{pp}[prop]{Proposition}

\newcommand\reallywidehat[1]{\savestack{\tmpbox}{\stretchto{  \scaleto{    \scalerel*[\widthof{\ensuremath{#1}}]{\kern-.6pt\bigwedge\kern-.6pt}    {\rule[-\textheight/2]{1ex}{\textheight}}  }{\textheight}}{0.5ex}}\stackon[1pt]{#1}{\tmpbox}}

\begin{document}

\title{Stationary solutions for the fractional Navier-Stokes-Coriolis system\\ in Fourier-Besov spaces}
\author{
{Leithold L. Aurazo-Alvarez
}{
\thanks{{
{\emph{Email address}}: leithold@unicamp.br}
\newline
{The author was supported by Cnpq (Brazil) and by University of Campinas, SP, Brazil.
 }}} \\
{\small 
IMECC-Department of Mathematics, University of Campinas} \\
{\small Campinas, SP, CEP 13083-859, Brazil.}}
\date{}
\maketitle

\begin{abstract}	
In this work we prove the existence of stationary solutions for the tridimensional fractional Navier-Stokes-Coriolis in critical Fourier-Besov spaces. We first deal with the non-stationary fractional Navier-Stokes-Coriolis and in this framework we get the existence of stationary solutions. Also we state a kind of stability of these non-stationary solutions which applied to the stationary case permits to conclude that, under suitable conditions, non-stationary solutions converge to the stationary ones when the time goes to infinity. Finally we establish a relation between the external force and the Coriolis parameter in order to get a unique solution for the stationary system. 

\end{abstract}
\renewcommand{\abstractname}{Abstract}

\section{Introduction}
In this work we are interested in looking for stationary solutions for the tridimensional fractional Navier-Stokes-Coriolis system

\begin{equation*}
(FNSC)\,\,\,\left\{
\begin{split}
&  \partial_{t}u+\nu(-\Delta)^{\alpha}u+u\cdot\nabla u+\Omega e_{3}\times u+\nabla p=F,\\
&  \mbox{div}\,u=0, \,\,\mbox{for}\,\,(x,t)\in\mathbb{R}^{3}\times(0,\infty
), \,\,\mbox{and}\\
&  u(x,0)=u_{0}(x),\,\,\mbox{for}\,\,x\in
\mathbb{R}^{3},
\end{split}
\right.  \label{SemigroupSystem}%
\end{equation*}
where $u = u(t, x) = (u_1(t, x), u_2(t, x), u_3(t, x))$ denotes the velocity vector field and $p = p(t, x)$ the scalar pressure of the fluid particle passing at the point $x = (x_1, x_2, x_3) \in \mathbb{R}^{3}$, $u_0(x)$ is the given initial velocity vector field and $F$ is the external force. The Coriolis parameter is denoted by $\Omega \in \mathbb{R}$ which is proportional to the angular velocity of the rotation with respect to the vertical unit vector $e_3 = (0, 0, 1)$, the kinematic viscosity coefficient is denoted by $\nu>0$ and the symbol $\times$ is reserved for the exterior product of vectors. Stationary solutions with homogeneity appear to play an important rol in the regularity of non-linear partial differential equations which are physical models  of reality phenomena, as one can observe in the regularity theory for harmonic maps.

Now we shall review some studies involving this kind of solutions in the entire space, in exterior domain or bounded domains in order to give an overview of the problems related with the above system. Leray (\cite{LerayEtudeHydrodynamique}, 1933), Fujita (\cite{FujitaExistenceRegularitySteady}, 1961) and Finn (\cite{FinnExteriorStationary}, 1965)
studied the existence of solutions for the exterior stationary problem of the Navier-Stokes system with $F=\nabla \cdot f$ and a null prescribed vector in the class of the finite Dirichlet integral $\displaystyle{\int_{\Omega}} \mid \nabla u(x)\mid^{2}\,dx<\infty$ and in the class of vectors satisfying 

\begin{equation*}
\displaystyle{\sup_{x\in\Omega}}\mid x\mid \mid u(x)\mid + \displaystyle{\sup_{x\in\Omega}}\,\frac{\mid x\mid^{2}}{\log\mid x\mid} \mid \nabla u(x)\mid < \infty.     
\end{equation*}
They mostly assumed the rapidly decaying at infinity of the tensor $F$. In particular, Finn (\cite{FinnExteriorStationary}, 1965) proved the existence of physically reasonable solutions in the case in which the data on the associated closed surface are close to a prescribed vector. This kind of solutions sketch a called ``wake region'' of fluid behind the surface which tends to follow the same surface along the opposite prescribed vector. The obtained solution tends to its limit at infinity in the order $\mid x\mid^{-1}$ in the wake region and in the order $\mid x\mid^{-2}$ in any different direction to the prescribed vector. Heywood (\cite{HeywoodStationaryNS}, 1970) studied the initial boundary-value problem for the Navier-Stokes system where the initial data is close to the solution of the associated stationary Navier-Stokes system. He proved that when the stationary solution satisfies an asymptotic estimate, the problem has a global unique solution and this solution converges to the associated stationary solution as time tends to infinity. Chen (\cite{ChenLnSolutionsStationary}, 1993) proved that the Navier-Stokes system in $\mathbb{R}^{n}$, for $n\geq 3$, and the associated stationary Navier-Stokes system have smooth solutions in $L^{n}$ and the non-stationary solution is arbitrarily close in the $L^{n}$-norm to the stationary solution as time goes to infinity. Kozono and Yamazaki (\cite{KozonoYamazakiStabilityStationaryMorrey}, 1995) proved that when the external force satisfies a certain condition there exists a unique small solution for the stationary Navier-Stokes system in Morrey spaces, moreover that stationary solutions are stable in the same Morrey space. Novotny and Padula (\cite{NovotnyPadulaNoteondecay}, 1995), Borchers and Miyakawa (\cite{BorcherdMiyakawaStabilityExteriorStationary}, 1995) and Galdi and Padula (\cite{GaldiPadulaExistenceSteadyObstacle}, 1991), by using the potential theory in hydrodynamics, studied the stationary Navier-Stokes system in the class

\begin{equation*}
\displaystyle{\sup_{x\in\Omega}}\mid x\mid^{n-2} \mid u(x)\mid + \displaystyle{\sup_{x\in\Omega}}\,\mid x\mid^{n-1} \mid \nabla u(x)\mid < +\infty,    
\end{equation*}
whenever the expression $\displaystyle{\sup_{x\in\Omega}}\mid x\mid^{n-1} \mid f(x)\mid + \displaystyle{\sup_{x\in\Omega}}\,\mid x\mid^{n} \mid \nabla f(x)\mid$ is small enough, where $F=\nabla \cdot f$. Le Jan and Sznitman (\cite{LeJanSznitmanStochastic cascadesNS}, 1997) studied the incompressible Navier-Stokes equations in $\mathbb{R}^{3}$ and introduce a probabilistic interpretation of the Navier-Stokes system. They, looking for the Fourier representation of the system, establish a non-linear integral equation for the Fourier transform of the Laplacian of the velocity, this integral equation can be interpreted by mean of a critical branching process (called stochastic cascade) and a composition rule along the associated tree and is based in a Markovian kernel with suitable properties. Moreover, in order to obtain existence and uniqueness results they obtain a ``domination principle'' for the associated non-linear integral equation. Tian and Xin (\cite{TianXinOnePointSingular}, 1998)
looking for specific solutions and solving a system of 
second order ordinary differential equations with four 
unknown functions, constructed a one-parameter family of 
explicit smooth solutions of the Navier-Stokes equations on 
$\mathbb{R}^{3}-\{x_0\}$, here $x_0$ is any given point; these
solutions solving the stationary Navier-Stokes system have 
some properties, for instance, they are axisymmetric, homogeneous of
degree $-1$ (called one-point singular solution), solve the
self-similar form for that system and proved the uniqueness in the 
class of axisymmetric flows. Kozono and Yamazaki (\cite{KozonoYamazakiExteriorProblemStationary}, 1998), using a functional analysis approach by mean of the $L^{r}$-theory for the associated Stokes equations, obtained a more larger class of external forces than rapidly decaying at infinity forces and also studied the behaviour of $u(x)$ and $\nabla u(x)$ as $\mid x\mid \rightarrow \infty$. They used the linearized method in Lorentz spaces $L_{n,\infty}$, considering $f\in L_{n/2,\infty}(D)$, where $D$ is an exterior domain in $\mathbb{R}^{n}$, $n\geq 3$ and $F=\nabla\cdot f$; thus they construct a solution $u$ of the stationary system with $u\in L_{n,\infty}(D)$ and 
$\nabla u\in L_{n/2,\infty}(D)$, being these norms invariant by the scaling of the system. Yamazaki (\cite{YamazakiWeakLnspaceTimedependentforce}, 2000) studied the Navier-Stokes system with time-dependent external force for $t\in\mathbb{R}$ as well as $t\in \mathbb{R}^{+}$, he studied that system considering the domain being the whole space, the half space or an exterior domain of dimension $n\geq 3$. In that work also are given sufficient conditions on the external force in order to get the existence of a unique existence of small solution in the weak-$L^{n}$ space which are bounded globally in time. By using the one-norm approach, Cannone and Karch (\cite{CannoneKarchSmoothorsingularNS}, 2004) proved the existence of singular solutions in the pseudomeasure space for the incompressible Navier-Stokes system with singular external forces (for instance the Dirac delta). This approach permits to consider the solutions obtained by Landau (\cite{LandauNewExactSolutionNS}, 1944) and by Tian and Xin (\cite{TianXinOnePointSingular}, 1998). It also was proved the asymptotic stability of small solutions including stationary ones. By the study of eigenvalues and eigenfunctions of the associated Coriolis operator, Chemin et al (\cite{CheminetalMGeophysics}, 2006) derived dispersion estimates for a linearized version of the Navier-Stokes-Coriolis system in order to show the existence of global solutions for the non-stationary Navier-Stokes-Coriolis system. Here they considered the decaying data case. Kim and Kozono (\cite{KimKozonoRemovableIsolatedSingularity}, 2006) considered the stationary Navier-Stokes equations in a non-empty open subset of $\mathbb{R}^{n}$, with $n\geq 3$, and proved that smooth solutions of the Navier-Stokes equations in the domain $B_{R}(0)-\{0\}$ are smooth solutions in the domain $B_{R}$, whenever the external force $F$ is smooth in $B_{R}(0)$ and the solution $u$ satisfies either conditions

\begin{equation*}
u\in L^{n}(B_{R}(0))\,\,\mbox{or}\,\, \mid u(x)\mid =o(\mid x\mid^{-1})\,\,\mbox{as}\,\, x\rightarrow 0.
\end{equation*}
 These kind of results are known as removable isolated singularity theorems for smooth solutions of the Navier-Stokes equations (NS). Ferreira and Villamizar-Roa (\cite{Ferreira-Villamizar-R-Micropolar-distributions-2007}, 2007) proved the existence of a unique mild solution for the generalized tridimensional micropolar system for small initial data in ${\cal PM}^{a}$-spaces and in this framework are also considered stationary solutions. They also established a result of asymptotic stability, similar to the ones obtained by Cannone and Karch (\cite{CannoneKarchSmoothorsingularNS}, 2004). Ferreira and Villamizar (\cite{Ferreira-Villamizar-R-Existence-convection-pseudomeasure-2008}, 2008 ) studied a $n$-dimensional generalized convection problem with gravitacional field depending of the space variable, in the context of ${\cal PM}^{a}$-spaces. They obtained local and global well-posedness result for this system and, in particular, for the  B\'enard problem, which correspond to a generalized Newtonian gravitational field. They also study regularizations properties for solutions. Bjorland and Schonbek (\cite{BjorlandSchonbekExistenceStabilitySteadyNS}, 2009), by considering the shape of the initial data near the origin in Fourier variables (i.e. low frequencies) and taking a small external force in a natural norm, obtained an energy decaying rate of solutions for an associated parabolic partial differential equation in the whole space and then with this decaying property it was possible rigorously integrate the parabolic solution in time in order to establish the existence of a finite $L^{2}(\mathbb{R}^{3})$-norm stationary solution. They also proved the uniqueness in the class of solutions with finite energy and
by using the Fourier splitting method, established that a finite energy perturbations to the stationary solution return to the initial stationary solution, this kind of stability is known as non-linearly stable. Konieczny and Yoneda (\cite{KoniecznyYonedaDispersiveStationaryNSC}, 2011) studied the dispersive effect of the Coriolis force for the stationary and non-stationary Navier-Stokes equations. For that aim was introduced the Fourier-Besov spaces which permits, by the explicit description of the semigroup in Fourier variable, to show the rol of the Coriolis parameter in the solution of the system. For some indexes, these Fourier-Besov spaces contain homogeneous functions of negative degree. One of the main results of that work is the existence of a unique solution for the stationary Navier-Stokes system with the Coriolis force for arbitrarily large external force, whenever the Coriolis force is choosen sufficiently large. Ferreira and Lima (\cite{Ferreira-Lima-Selfsimilar-active-scalar-FBM-2014}, 2014)  studied a family of dissipative active scalar equation considering fractional diffusion with sub-critical values. In order to do that and motivated by Fourier-Besov spaces, they introduce the Fourier-Besov-Morrey spaces, which for some fixed parameter coincide with the Fourier-Besov spaces.  For some indexes, both of these spaces allow to consider large $L^{2}$-norms and homogeneous functions of negative degree, for the initial data. They established the well-posedness for the equation for small initial data in a critical Fourier-Besov-Morrey space by mean of a one-norm approach. They also provided an analysis of asymptotic stability for solutions. In general, are important function spaces which contain singular functions. To mention one, Ferreira and Precioso (\cite{Ferreira-Precioso-Existence-micropolar-Besov-Morrey-2013}, 2013) studied the $3D$-micropolar system and obtained a local well-posedness result in critical Besov-Morrey spaces, which contain highly singular functions and measures whose support could be a point, a filament or a surface. Almeida, Ferreira and Lima (\cite{Almeida-Ferreira-Lima-Uniform-NSC-FBM-2017}, 2017) also studied the Navier-Stokes-Coriolis system in the framework of critical Fourier-Besov-Morrey spaces and obtained a uniform (the smallness conditions for the initial data does not depends of the Coriolis parameter) global well-posedness result for this system. Kaneko, Kozono and Shimizu (\cite{KanekoKozonoShimizuStationaryNSBesov}, 2019) studied the stationary Navier-Stokes system in $\mathbb{R}^{n}$, for $n\geq 3$, and they obtained the existence, uniqueness and regularity of solutions in the associated scaling invariante homogeneous Besov space $B^{-1+\frac{n}{p}}_{p,q}(\mathbb{R}^{n})$, for $1\leq p<n$ and $1\leq q\leq \infty$, for an small external force $F\in B^{-3+\frac{n}{p}}_{p,q}(\mathbb{R}^{n})$ , where restriction for the indexes are taken because of the bilinear estimates for the paraproduct formula and the imbedding theorem in Besov spaces. It is important to observe that the space $B^{-3+\frac{n}{p}}_{p,q}(\mathbb{R}^{n})$ contains homogeneous functions of degree $-3$, for the external force, and the space $B^{-3+\frac{n}{p}}_{p,q}(\mathbb{R}^{n})$ contains homogeneous functions of degree $-1$, for the velocity, so that the existence result also give rise to self-similar solutions for this system. Ferreira and the author (\cite{Aurazo-Ferreira-Global-BC-Stratification-FBM-2021}, 2021) studied the tridimensional fractional Boussinesq-Coriolis system with stratification 
in the context of critical Fourier-Besov-Morrey spaces. We proved a global well-posedness result for small initial data in this spaces, for a range of values for the exponent of minus the Laplacian. There, we also consider the critical case for this system and for the tridimensional fractional Navier-Stokes-Coriolis system.  

Results given in the current work have two objectives: The first one is to expose a one-norm approach for the global well-posedness of the fractional Navier-Stokes-Coriolis system in critical Fourier-Besov spaces. Thus we extend some results obtained by Konieczny and Yoneda (\cite{KoniecznyYonedaDispersiveStationaryNSC}, 2011) for the non-stationary case as well as the stationary one. In particular the existence of a unique small stationary solutions for the (FNSC)-system for any arbitrarily large external force, whenever the Coriolis parameter is taken large enough. The second one is to describe the behaviour at infinity of solutions for the (FNSC)-system for a range o values $\frac{1}{2}<\alpha<\frac{1}{4}\left(5-\frac{3}{p}\right)$ in critical Fourier-Besov spaces $\dot{FB}^{4-2\alpha-\frac{3}{p}}_{p,q}$, for $p>1$ and $1\leq q\leq \infty$; thus, for $\alpha=1$ and $\Omega=0$, we complement the approach given by Cannone and Karch (\cite{CannoneKarchSmoothorsingularNS}, 2004) for ${\cal PM}^{2}$-spaces, which correspond to $\dot{FB}^{4-2\alpha-\frac{3}{p}}_{p,q}$ with $\alpha=1$ and $p=q=\infty$. For instance, it is valid that, under suitable conditions, non-stationary solutions for the (FNSC)-system tends to the stationary ones as time goes to infinity. In particular, all results given here are valid for homogeneous Sobolev spaces $\dot{H}^{\frac{5}{2}-2\alpha}$, for the range $\frac{1}{2}<\alpha<\frac{7}{8}$. Moreover, for the case $\alpha=1$ and $\Omega=0$, we also can claim that we obtain a partial counterpart for Theorem 1.1 obtained by Kaneko, Kozono and Shimizu (\cite{KanekoKozonoShimizuStationaryNSBesov}, 2019) in Besov-spaces.

This work is organized as follow: In Section \ref{preliminar} we recall function spaces, the associated semigroup for the (FNSC)-system, the notion of mild solution and introduce some tools related to our approach. In Section \ref{main} we state the main theorems of this work. Finally in Section \ref{stationary} we prove these main results.

\section{Preliminaries}\label{preliminar}
 In this section we recall some of the used function spaces and properties related to, we state the fractional Navier-Stokes-Coriolis semigroup, the notion of mild solutions for the (FNSC)-system and the (SFNSC)-system; we recall the paraproduct formula, the fixed point lemma, which we shall apply in order to prove the existence of a unique mild solution, and finally some related time-dependent spaces in order to establish the language for this work.
 \subsection{Fourier-Besov spaces}
Fourier-Besov spaces were introduced in the work of Konieczny and Yoneda (\cite{KoniecznyYonedaDispersiveStationaryNSC}, 2011) to study the dispersive effect of the Coriolis force for both the stationary and the non-stationary Navier-Stokes-Coriolis system. They consider a radially symmetric function $\varphi\in {\cal S}(\mathbb{R}^{3})$ supported in the annulus $\{\xi\in\mathbb{R}^{3};\frac{3}{4}\leq \mid \xi\mid \leq \frac{8}{3}\}$ such that 

\begin{equation*}
\displaystyle{\sum_{j\in\mathbb{Z}}}\varphi(2^{-j}\xi)=1\,\,\mbox{for all}\,\,\xi\neq 0.    
\end{equation*}
With this function $\varphi\in {\cal S}(\mathbb{R}^{3})$ they introduce two associated functions

\begin{equation*}
    \varphi_j(\xi)=\varphi(2^{-j}\xi)\,\,\mbox{and}\,\,\psi_j(\xi)=\displaystyle{\sum_{k\leq j-1}}\varphi_k(\xi).
\end{equation*}
After that, they introduce the standard localization operators:

\begin{equation*}
\Delta_{j}f=\varphi_{j}(D)f,\,\,\,S_{j}f=\displaystyle{\sum_{k\leq j-1}}\Delta_{k}f=\psi_{j}(D)f,\,\,\,\mbox{for every}\,\,j\in\mathbb{Z}.
\end{equation*}
These operators satisfy the following properties:

\begin{equation*}
\Delta_j\Delta_kf=0,\,\,\mbox{if}\,\mid j-k\mid \geq 2, \,\,\mbox{and}    
\end{equation*}

\begin{equation*}
\Delta_j(S_{k-1}f \Delta_kf)=0, \,\,\mbox{if}\,\mid j-k\mid \geq 5.    
\end{equation*}
\begin{defi}
The homogeneous Fourier-Besov spaces are defined in both the following cases:
\begin{itemize}
    \item [(i)] For $1\leq p\leq \infty$, $1\leq q< \infty$ and $s\in \mathbb{R}$,
    
    \begin{equation*}
    \dot{FB}^{s}_{p,q}(\mathbb{R}^{n})=\{
    f\in {\cal S}': \hat{f}\in L^{1}_{loc}, \parallel f\parallel_{\dot{FB}^{s}_{p,q}(\mathbb{R}^{n})}=
    \left( \displaystyle{\sum_{k\in\mathbb{Z}}}\,2^{ksq}\parallel \varphi_k\hat{f}\parallel_{L^{p}(\mathbb{R}^{n})}^{q} \right)^{1/q}<+\infty
    \}
    \end{equation*}
     \item [(ii)] For $1\leq p\leq \infty$ and $s\in \mathbb{R}$,
     
\begin{equation*}
    \dot{FB}^{s}_{p,\infty}(\mathbb{R}^{n})=\{
    f\in {\cal S}': \hat{f}\in L^{1}_{loc}, \parallel f\parallel_{\dot{FB}^{s}_{p,\infty}(\mathbb{R}^{n})}=\displaystyle{\sup_{k\in\mathbb{Z}}}\,2^{ks}\parallel\varphi_k\hat{f}\parallel_{L^{p}(\mathbb{R}^{n})}<+\infty\}
\end{equation*}
\end{itemize}
\begin{obs}
The space $\dot{FB}_{p,q}^{s}$ contains homogeneous
functions of degree $-d=s-n+\frac{n}{p}$.
\end{obs}

\end{defi}
Next lemma describes some useful assertions we use in this work related to $L^{p}$-spaces and Fourier-Besov spaces.

\begin{lema}
Let $s_{1},s_{2}\in\mathbb{R}$, $1\leq p_{1},p_{2}<+\infty$.

\begin{itemize}

\item[(i)] (Bernstein-type inequality) Suppose that $p_{2}\leq p_{1}$ and $A>0$. If $\mbox{supp}\,(\hat{f})\subset\{\xi\in\mathbb{R}^{n};\mid\xi\mid\leq A2^{j}\}$, then

\begin{equation*}
\parallel\xi^{\beta}\hat{f}\parallel_{p_{2}}\leq C2^{j\mid\beta
\mid+j(\frac{n}{p_{2}}-\frac{n}{p_{1}})}\parallel\hat
{f}\parallel_{p_{1}},
\end{equation*}
where $\beta$ is the multi-index, $j\in\mathbb{Z}$, and $C>0$ is a constant
independent of $j,\xi$ and $f$.

\item[(ii)] (Sobolev-type embedding) For $p_{2}\leq p_{1}$ and $s_{2}\leq
s_{1}$ satisfying $s_{2}+\frac{n}{p_{2}}=s_{1}+\frac{n}{p_{1}%
}$, we have the continuous inclusion
\[
\dot{FB}_{p_{1},r_{1}}^{s_{1}}\subset\dot{FB}_{p_{2},r_{2}}^{s_{2}},
\]
for all $1\leq r_{1}\leq r_{2}\leq\infty$.
\end{itemize}
\end{lema}

\begin{obs}
	When indexes $p$ and $q$ are equals there is an equivalent norm on $\dot{FB}^{s}_{p,p}$, 
	
	\begin{equation*}
	\parallel f\parallel_{\dot{FB}^{s}_{p,p}}\sim
	\left(\displaystyle{\int_{\mathbb{R}^{d}}} \mid\xi\mid^{sp}
		\mid \hat{f}(\xi)\mid^{p}d\,\xi \right)^{\frac{1}{p}}.
	\end{equation*}
\end{obs}
\subsection{Fractional Navier-Stokes-Coriolis semigroup and mild solutions}
Considering the following Stokes problem with the Coriolis force,

\begin{equation*}
\,\,\,\left\{
\begin{split}
&  \partial_{t}u+\nu(-\Delta)^{\alpha}u+\Omega e_{3}\times u+\nabla p=0,\\
&  \mbox{div}\,u=0, \,\,\mbox{for}\,\,(x,t)\in\mathbb{R}^{3}\times(0,\infty
), \,\,\mbox{and}\\
&  u(x,0)=u_{0}(x),\,\,\mbox{for}\,\,x\in
\mathbb{R}^{3},
\end{split}
\right.  \label{SemigroupSystem}%
\end{equation*}
it is possible to get the associated fractional Stokes-Coriolis semigroup, denoted by $S^{\alpha}_{\Omega}$ which in Fourier variables is given by

\begin{equation*}
(S^{\alpha}_{\Omega}u_0)^{\wedge}(t,\xi)=\cos(\Omega\frac{\xi_3}{\mid\xi\mid})e^{-\nu\mid\xi\mid^{2\alpha}t}I\hat{u}_0(\xi)+
\sin(\Omega\frac{\xi_3}{\mid\xi\mid})e^{-\nu\mid\xi\mid^{2\alpha}t}R(\xi)\hat{u}_0(\xi),
\end{equation*}
where $t\geq 0$, $\xi\in \mathbb{R}^{3}$, $I$ is the identity matrix and the matrix $R(\xi)$ is given by

\begin{equation*}
R(\xi)=\left(
\begin{matrix}
0 & \frac{\xi_3}{\mid\xi\mid} & -\frac{\xi_2}{\mid\xi\mid}\\
-\frac{\xi_3}{\mid\xi\mid} &  0 & \frac{\xi_1}{\mid\xi\mid}\\
\frac{\xi_2}{\mid\xi\mid} & -\frac{\xi_1}{\mid\xi\mid} & 0
\end{matrix}
\right). 
\end{equation*}
By the Duhammel's principle, the (FNSC)-system is 
equivalent to the integral equation given by 

\begin{equation}\label{mildform}
u(t)=S^{\alpha}_{\Omega}(t)u_0 + B(u,u)+\tilde{F}, 
\end{equation}
where the bilinear operator $B(\cdot, \cdot)$ is def{}ined via the Fourier 
transform as follows,

\begin{equation}\label{bilinearterm}
[B(v,w)]\textasciicircum(\xi,t)=
-\displaystyle{\int_{0}^{t}}[S^{\alpha}_{\Omega}(t-\tau)]\textasciicircum(\xi) [\mathbb{P}]\textasciicircum (\xi)[i\xi\cdot(v\otimes w)\textasciicircum(\xi,\tau)]d\,\tau.
\end{equation}
Here the elements of the matrix $[v\otimes w]\textasciicircum(\xi,\cdot)$ are given by 

\begin{equation*}
([v\otimes w]\textasciicircum)_{mk}=\hat{v}_m\ast \hat{w}_k(\xi,\cdot),\,\,
\mbox{for}\,\,1\leq m,k\leq 3.
\end{equation*}
Also, the term

\begin{equation*}
\tilde{F}(t)=\displaystyle{\int_{0}^{t}}S^{\alpha}_{\Omega}(t-\tau)\mathbb{P}F(\tau)\,d\tau    
\end{equation*}
 is given in Fourier-variables by
 
\begin{equation*}
(\tilde{F})^{\wedge}(t,\xi)=\displaystyle{\int_{0}^{t}}(S^{\alpha}_{\Omega})^{\wedge}(t-\tau,\xi)\hat{\mathbb{P}}(\xi)\hat{F}(\tau,\xi)\,d\tau,    
\end{equation*}
where the symbol for the projector operator $\hat{\mathbb{P}}(\xi)$ is bounded, for all $\xi\in\mathbb{R}^{3}$. 
A solution in an appropiated Banach space ${\cal X}$ for the integral equation (\ref{mildform}) is called a mild solution for the (FNSC)-system.
\begin{lema}\label{equivStationary}
Let be $u=u(x)\in \dot{FB}^{4-2\alpha-\frac{3}{p}}_{p,q}$ and $F=F(x)\in \dot{FB}^{4-4\alpha-\frac{3}{p}}_{p,q}$. The following expressions are equivalents.
\begin{itemize}
    \item [(i)] $u$ is a stationary mild solution of the (FNSC)-system , that it, 
    
\begin{equation*}
u=S^{\alpha}_{\Omega}(t)u-\displaystyle{\int_{0}^{t}}S^{\alpha}_{\Omega}(t-\tau)
\mathbb{P}\nabla\cdot(u\otimes u)\,d\tau+\displaystyle{\int_{0}^{t}}S^{\alpha}_{\Omega}(\tau)\mathbb{P}F\,d\tau,
\end{equation*}
for every $t>0$.
\item [(ii)] u satisfies the integral equation

\begin{equation*}
u=-\displaystyle{\int_{0}^{\infty}}S^{\alpha}_{\Omega}(\tau)\mathbb{P}\nabla\cdot
(u\otimes u)\,d\tau+\displaystyle{\int_{0}^{\infty}}S^{\alpha}_{\Omega}(\tau)\mathbb{P}F\,d\tau,
\end{equation*}
    where both the integrals must be undertood in the Fourier-variables, for almost $\xi\in\mathbb{R}^{3}$.
\end{itemize}
\end{lema}

\subsection{Paraproduct formula, abstract f{}ixed point lemma and
time-dependent spaces}
In this subsection we set the Bony's paraproduct formula, describe an abstract
f{}ixed point lemma and we recall suitable time-dependent spaces based on Fourier-Besov spaces which are the required solution spaces. 

Let $f,g\in\mathcal{S}^{\prime}(\mathbb{R}^{n})$, it is defined  the Bony's
paraproduct operator $T_{f}(\cdot)$ as the following expression

\begin{equation*}
T_{f}(g)=\displaystyle{\sum_{j\in\mathbb{Z}}}S_{j-1}f\Delta_{j}g, \label{Top}%
\end{equation*}
also the operator $R(\cdot,\cdot)$ is defined by 

\begin{equation*}
R(f,g)=\displaystyle{\sum_{j\in\mathbb{Z}}}\Delta_{j}f\tilde{\Delta}%
_{j}g\,\,\,\mbox{with}\,\,\,\tilde{\Delta}_{j}g=\displaystyle{\sum_{\mid
j^{\prime}-j\mid\leq1}}\Delta_{j^{\prime}}g. \label{Top2}%
\end{equation*}
The symbol of the operator $\tilde{\Delta}_{j}$ in (\ref{Top2}) is denoted by
$\tilde{\varphi}_{j}=\displaystyle{\sum_{\mid j'-j\mid\leq 1}}\varphi_{j'}$. By mean of these operators, we can put explicitly the product of $f$ by $g$ as

\begin{equation}\label{paraproduct}
fg=T_{f}(g)+T_{g}(f)+R(f,g). %
\end{equation}
Expression (\ref{paraproduct}) is known as Bony's paraproduct formula, see (\cite{BonyCalculSymbolique}, 1981) and (\cite{Lemarie2002}, 2002).
One of the main tools for this approach is the rol given to the next fixed point lemma (see \cite{Lemarie2002}).

\begin{lema}\label{fixedpoint} Let be $(X,\parallel\cdot\parallel)$ a Banach space and
$B:X\times X\rightarrow$ X a bilinear operator satisfying $\parallel
B(x_{1},x_{2})\parallel_{X}\leq K\parallel x_{1}\parallel_{X}\parallel x_{2}\parallel_{X}$, for all $x_{1},x_{2}$, where $K>0$ is a constant. If $0<\varepsilon<\frac
{1}{4K}$ and $\parallel y\parallel_{X}\leq\varepsilon$, then the equation
$x=y+B(x,x)$ has a solution in $X$. Moreover, this solution is unique in the
closed ball $\{x\in X;\parallel x\parallel_{X}\leq2\varepsilon\}$ and $\parallel x\parallel_{X}\leq2\parallel y\parallel_{X}$. The solution depends continuously on
$y$; in fact, if $\parallel\tilde{y}\parallel_{X}\leq\varepsilon$, $\tilde
{x}=\tilde{y}+B(\tilde{x},\tilde{x})$ and $\parallel\tilde{x}\parallel_{X}
\leq2\varepsilon$, then

\begin{equation*}
\parallel x-\tilde{x}\parallel_{X}\leq(1-4K\varepsilon)^{-1}\parallel y-\tilde
{y}\parallel_{X}. \label{fixedpointdependence}%
\end{equation*}
\end{lema}
Let us observe that, if we write $y=S_{\Omega}^{\alpha}(t)u_{0}+\tilde{F}(t)$ and $B(v,w)$ as in
(\ref{bilinearterm}) then the integral equation (\ref{mildform}) have the form
$u=y+B(u,u)$ required in Lemma \ref{fixedpoint}. Finally we set two time-dependent spaces. Let $1\leq p\leq\infty$,
$0<T\leq\infty$ and $I=(0,T)$. The Banach spaces $L^{p}(I;\dot{FB}%
_{q,r}^{s})$ and $\mathcal{L}^{p}(I;\dot{FB}_{q,r}^{s})$ are the
set of Bochner measurable functions from $I$ to $\dot{FB}_{q,r}^{s}$
with respective norms given by

\begin{equation*}
\parallel f\parallel_{L^{p}(I;\dot{FB}_{q,r}^{s})}%
=
\parallel \parallel f(\cdot,t)\parallel_{\dot{FB}_{q,r}^{s}}\parallel_{L^{p}(I)}
=\parallel
\displaystyle{\left(\displaystyle{\sum_{j\in\mathbb{Z}}}2^{jsr}\parallel
\varphi_{j}\hat{f}\parallel_{q}^{r}\right)^{1/r}\parallel_{L^{p}(I)}}
\end{equation*}
and

\begin{equation*}
\parallel f\parallel_{\mathcal{L}^{p}(I;FB_{q,r}^{s}%
)}
=\left(\displaystyle{ \sum_{j\in\mathbb{Z}}}\,2^{jsr}%
\parallel\varphi_{j}\hat{f}\parallel_{L^{p}(I;L^{q})}^{r}\right)^{1/r}
=\left(\displaystyle{ \sum_{j\in\mathbb{Z}}}\,2^{jsr}\
\parallel\parallel\varphi_{j}\hat{f}\parallel_{q}\parallel
_{L^{p}(I)}^{r}\right)^{1/r}.
\end{equation*}
Throughout this work we shall use the notation

\begin{equation*}
\mathcal{X}=:\mathcal{L}^{\infty}(I;FB_{p,q}^{4-2\alpha-\frac{3}{p}}).
\end{equation*}

\section{Main Results}\label{main}
Here we establish a well-posedness result for the fractional Navier-Stokes-Coriolis system (FNSC) by the one-norm approach in the space ${\cal X}$.
\begin{teo}\label{teorema1}
Let be $1< p\leq \infty$, $1\leq q\leq\infty$, 
 $\frac{1}{2}<\alpha<\frac{1}{4}(5-\frac{3}{p})$ and let be $K$ the constant given in Proposition \ref{constantK}. There is a constant $0<\varepsilon<\frac{1}{4K}$ independent of the 
Coriolis parameter such that, for each $u_0\in \dot{FB}^{4-2\alpha-\frac{3}{p}}_{p,q}$ satisfying

\begin{equation*}
\parallel u_0\parallel_{\dot{FB}^{4-2\alpha-\frac{3}{p}}_{p,q}}+
\parallel F\parallel_{{\cal L}^{\infty}(I;\dot{FB}^{4-4\alpha-\frac{3}{p}}_{p,q})}
<\varepsilon,
\end{equation*}
there is a global mild solution for the  $(FNSC)$-system in the space  ${\cal X}={\cal L}^{\infty}(I;\dot{FB}^{4-2\alpha-\frac{3}{p}}_{p,q})$. Moreover, this solution is the unique global solution in the ball 
$\{ u\in{\cal X};\parallel u\parallel_{\cal X}\leq 2\varepsilon \}$ and these solutions depends continuosly on initial data and external forces, as indicated by the fixed point lemma (Lemma \ref{fixedpoint}).
\end{teo}
The next theorem describes a kind of asymptotic behaviour of the system (FNSC), which applied to stationary solutions implies that non-stationary solutions of the (FNSC)-system tends toward the stationary ones, depending on the behaviour at infinity of the semigroup and the external force in Fourier-Besov norms. 

\begin{teo}\label{teorema2}
Under assumptions of Theorem \ref{teorema1} above we have the following asymptotic behavior of the system (FNSC). Let $u$ and $v$ be two solutions of the (FNSC)-system arising from Theorem \ref{teorema1},
 associated to the initial data $u_0$ and $v_0$ in $\dot{FB}^{4-2\alpha-\frac{3}{p}}_{p,q}$ and  external forces $F$ and $G$ 
 in ${\cal L}^{\infty}(I;\dot{FB}^{4-4\alpha-\frac{3}{p}}_{p,q})$, respectively.
 If
 
 \begin{equation}\label{F1}
 \displaystyle{\lim_{t\rightarrow\infty}}
 \parallel S^{\alpha}_{\Omega}(t)(u_0-v_0)
 \parallel_{\dot{FB}^{4-2\alpha-\frac{3}{p}}_{p,q}}=0\,\,\mbox{and}\,\,
   \displaystyle{\lim_{t\rightarrow\infty}}
 \parallel F(t)-G(t)
  \parallel_{\dot{FB}^{4-4\alpha-\frac{3}{p}}_{p,q}}=0,
 \end{equation}
 then 
 
 \begin{equation*}
  \displaystyle{\lim_{t\rightarrow\infty}}
  \parallel u(\cdot,t)-v(\cdot,t)
  \parallel_{\dot{FB}^{4-4\alpha-\frac{3}{p}}_{p,q}}=0.
 \end{equation*}
\end{teo} 
Next theorem it is an existence result of stationary solutions for the system (SFNSC) by the one-norm approach, depending of the smallness of the external force $F$. 
\begin{teo}\label{teorema3}
Let be $1< p\leq \infty$, $1\leq q\leq\infty$, 
 $\frac{1}{2}<\alpha<\frac{1}{4}(5-\frac{3}{p})$. There exists two constant $C>0$  and $\varepsilon>0$ such that for a given external force $F\in \dot{FB}^{4-4\alpha-\frac{3}{p}}_{p,q}$ such
that $\parallel F\parallel_{\dot{FB}^{4-4\alpha-\frac{3}{p}}_{p,q}}<\frac{\varepsilon \nu}{C}$
there is a stationary solution for the (SFNSC)-system in the space $\dot{FB}^{4-2\alpha-\frac{3}{p}}_{p,q}$ with external force $F$.
Moreover, this solution is unique in the ball $\{ u\in \dot{FB}^{4-2\alpha-\frac{3}{p}}_{p,q};
\parallel u\parallel_{\dot{FB}^{4-2\alpha-\frac{3}{p}}_{p,q}} \leq 2\varepsilon\}$.
\end{teo}
Next result establish that for any external force, it is possible to consider a large enough Coriolis parameter in order to obtain a unique stationary solution for the (SFNSC)-system.
\begin{teo}\label{teorema4}
Let $1<p<\infty$ and $\frac{1}{2}<\alpha< \frac{1}{4}(5-\frac{3}{p})$. If we consider an external force $F\in \dot{FB}^{4-4\alpha-\frac{3}{p}}_{p,p}$ then there exist $\Omega_0$ such that, for all Coriolis parameter $\Omega\in \mathbb{R}$ with $\mid \Omega\mid \geq \Omega_0$, there is a solution $u\in \dot{FB}^{4-2\alpha-\frac{3}{p}}_{p,p}(\mathbb{R}^{3})$ for the (SFNSC)-system that is the unique solution such that $\parallel u\parallel_{\dot{FB}^{4-2\alpha-\frac{3}{p}}_{p,p}(\mathbb{R}^{3})}\leq 2\varepsilon$.
\end{teo}
\begin{obs}
Under a suitable assumption on $F$ a similar result it is valid for the external force $F\in \dot{FB}^{4-4\alpha-\frac{3}{p}}_{p,\infty}$ (see condition \ref{condicionF} below).
\end{obs}


\section{Main estimates for the system in Fourier-Besov spaces  }\label{stationary}
In this section we provide a one-norm approach in order to establish the global existence of mild solutions for the fractional Navier-Stokes-Coriolis system in the context of Fourier-Besov spaces.
In this context we shall obtain the existence of stationary solutions for the considered system. 
Ideas for the proof of next lemmas are based on computations developed in (\cite{CannoneKarchSmoothorsingularNS}, 2004) and (\cite{KoniecznyYonedaDispersiveStationaryNSC}, 2011). 

We start with the persistence of regularity of the semigroup applied to any element in Fourier-Besov spaces and the time weak continuity of this map, looking as a path in the same Fourier-Besov space.
Let us recall that ${\cal X}={\cal L}^{\infty}\left(I;\dot{FB}^{4-2\alpha-\frac{3}{p}}_{p,q}\right)$, where $I$ denotes the interval $[0,\infty)$.
\begin{lema}\label{lemmasemFC}
If $u_0\in \dot{FB}^{4-2\alpha-\frac{3}{p}}_{p,q}$ then

\begin{equation*}
S^{\alpha}_{\Omega}(t)u_0\in \dot{FB}^{4-2\alpha-\frac{3}{p}}_{p,q},\,\,\mbox{for all}\,\,t>0.
\end{equation*}
Moreover $S^{\alpha}_{\Omega}(\cdot)u_0\in 
{\cal L}^{\infty}(I;\dot{FB}^{4-2\alpha-\frac{3}{p}}_{p,q})$ and for $1<p<+\infty$, $1<q<+\infty$ and $\alpha>\frac{1}{2}$ it is valid the weak convergence

\begin{equation*}
    S^{\alpha}_{\Omega}(t)u_0\rightharpoonup S^{\alpha}_{\Omega}(s)u_0\,\,\mbox{as}\,\,t\rightarrow s\,\,\mbox{for each}\,\, s\in I.
\end{equation*}
\end{lema}
\begin{proof}
From def{}inition of the norm in $\dot{FB}^{4-2\alpha-\frac{3}{p}}_{p,q}$, considering the bound of the matrix $R(\xi)$ and since $e^{-\nu t\mid\xi\mid^{2\alpha}}\leq 1$, we have the estimative
\begin{eqnarray*}
\parallel S^{\alpha}_{\Omega}(t)u_0\parallel_{\dot{FB}^{4-2\alpha-\frac{3}{p}}_{p,q}}&=&
\parallel 2^{j(4-2\alpha-\frac{3}{p})}\parallel \varphi_j
\left(S^{\alpha}_{\Omega}(t) \right)^{\wedge}\hat{u}_0\parallel_{L^{p}}
\parallel_{l^{q}}\\
&\leq& C 
\parallel 2^{j(4-2\alpha-\frac{3}{p})}\parallel \varphi_j
e^{-\nu t\mid\xi\mid^{2\alpha}}\hat{u}_0\parallel_{L^{p}}
\parallel_{l^{q}}\\
&\leq &
C\parallel u_0\parallel_{\dot{FB}^{4-2\alpha-\frac{3}{p}}_{p,q}},
\end{eqnarray*} 
where the constant $C>0$ does not depends on the Coriolis parameter $\Omega$. Therefore $S^{\alpha}_{\Omega}(t)u_0\in \dot{FB}^{4-2\alpha-\frac{3}{p}}_{p,q}$, for all $t>0$. We also can get the following estimate in the space ${\cal L}^{\infty}
(I; \dot{FB}^{4-2\alpha-\frac{3}{p}}_{p,q})$ ,
\begin{eqnarray*}
\parallel S^{\alpha}_{\Omega}(\cdot)u_0\parallel_{
{\cal L}^{\infty}
(I; \dot{FB}^{4-2\alpha-\frac{3}{p}}_{p,q})}&=&
\parallel 2^{j(4-2\alpha-\frac{3}{p})}
\parallel \varphi_j\left(S^{\alpha}_{\Omega}(t)\right)^{\wedge}\hat{u}_0
\parallel_{L^{\infty}(I;L^{p})}\parallel_{l^{q}}\\
&\leq &C\parallel u_0\parallel_{\dot{FB}^{4-2\alpha-\frac{3}{p}}_{p,q}},
\end{eqnarray*}	
where $C>0$ is a constant which does not depends on the Coriolis parameter 
$\Omega$.
Now let us show the weak continuity at the time $s=0$. For each 
$\psi\in{\cal S}$, we use the Plancherel formula in order to get 
\begin{eqnarray*}
\mid \langle S^{\alpha}_{\Omega}(t)u_0-u_0,\psi\rangle\mid&=&
\mid \langle \left(S^{\alpha}_{\Omega}(t)\right)^{\wedge}\hat{u}_0-\hat{u}_0,\hat{\psi}\rangle\mid\\
&\leq&
\displaystyle{\int_{\mathbb{R}^{3}}}
	\mid 1-e^{-\nu t\mid\xi\mid^{2\alpha}}\cos\left(\Omega\frac{\xi_3}{\mid\xi\mid}t\right)\mid
	\parallel I\parallel\mid \hat{u}_0\mid\mid\hat{\psi}(\xi)\mid\,d\xi+\\
& & 
\displaystyle{\int_{\mathbb{R}^{3}}}
\mid e^{-\nu t\mid\xi\mid^{2\alpha}}\sin\left(\Omega\frac{\xi_3}{\mid\xi\mid}t\right)\mid \parallel R(\xi)\parallel\mid \hat{u}_0\mid\mid\hat{\psi}(\xi)\mid\,d\xi\\
&\leq & Ct
\left(
\displaystyle{\int_{\mathbb{R}^{3}}}
\displaystyle{\sup_{0\leq\theta\leq t}}
\mid \frac{d}{d\theta} e^{-\nu \theta\mid\xi\mid^{2\alpha}}\cos\left(
\frac{\Omega\xi_3}{\mid\xi\mid}\theta\right)\mid 
	\mid\hat{u}_0(\xi)\mid \mid \hat{\psi}(\xi)\mid\,d\xi+\right.\\	 
& & \left. 
\displaystyle{\int_{\mathbb{R}^{3}}}
\displaystyle{\sup_{0\leq\theta\leq t}}
\mid \frac{d}{d\theta} e^{-\nu \theta\mid\xi\mid^{2\alpha}}\sin\left(
\frac{\Omega\xi_3}{\mid\xi\mid}\theta\right)\mid 
\mid\hat{u}_0(\xi)\mid \mid \hat{\psi}(\xi)\mid\,d\xi
\right)
\end{eqnarray*}
Considering estimates 

\begin{equation*}
\mid \frac{d}{d\theta} e^{-\nu \theta\mid\xi\mid^{2\alpha}}\cos\left(
\frac{\Omega\xi_3}{\mid\xi\mid}\theta\right)\mid \leq 
\mid\Omega\mid e^{-\nu\theta\mid\xi\mid^{2\alpha}}+
\nu\mid\xi\mid^{2\alpha}e^{-\nu\theta\mid\xi\mid^{2\alpha}}
\end{equation*}
and  

\begin{equation*}
\mid \frac{d}{d\theta} e^{-\nu \theta\mid\xi\mid^{2\alpha}}\sin\left(
\frac{\Omega\xi_3}{\mid\xi\mid}\theta\right)\mid \leq 
\mid\Omega\mid e^{-\nu\theta\mid\xi\mid^{2\alpha}}+
\nu\mid\xi\mid^{2\alpha}e^{-\nu\theta\mid\xi\mid^{2\alpha}},
\end{equation*}
we have

\begin{equation}\label{des4.2A}
\mid \langle S^{\alpha}_{\Omega}(t)u_0-u_0,\psi\rangle\mid\leq C_{\Omega,\nu}\, \left(I_1+I_2\right)t, 
\end{equation}
where 

\begin{equation*}
I_1=
\displaystyle{\int_{\mathbb{R}^{3}}}
\mid\hat{u}_0(\xi)\mid \mid\hat{\psi}(\xi)\mid\,d\xi\,\,\mbox{and}\,\,
I_2=
\displaystyle{\int_{\mathbb{R}^{3}}}\mid \xi\mid^{2\alpha}
\mid\hat{u}_0(\xi)\mid \mid\hat{\psi}(\xi)\mid\,d\xi.
\end{equation*}
We claim that $I_1$ and $I_2$ are finite. We first estimate $I_1$:
Since $\displaystyle{\sum_{j\in\mathbb{Z}}}\varphi_j=1$, for each $\xi\in\mathbb{R}^{3}-\{0\}$, we have
\begin{eqnarray*}
I_1&=&
 \displaystyle{\sum_{j\in\mathbb{Z}}}
\displaystyle{\int_{2^{j}\cal{C}}}
\varphi_j(\xi) \mid\hat{u}_0(\xi)\mid \mid\hat{\psi}(\xi)\mid\,d\xi\\
&\leq&
 \displaystyle{\sum_{j\in\mathbb{Z}}}
\parallel \varphi_j\hat{u}_0\parallel_{L^{p}}
\left(
\displaystyle{\int_{2^{j}\cal{C}}}\mid \hat{\psi}(\xi)\mid^{p'}\,d\xi
\right)^{\frac{1}{p'}}\\
&\leq & 
 \displaystyle{\sum_{j\in\mathbb{Z}}}\,
 2^{j(4-2\alpha-\frac{3}{p})} 
 \parallel \varphi_j\hat{u}_0\parallel_{L^{p}}
 \left(
 \displaystyle{\int_{2^{j}\cal{C}}}
 \mid\xi\mid^{(-4+2\alpha+\frac{3}{p})p'}
 \mid \hat{\psi}(\xi)\mid^{p'}\,d\xi
 \right)^{\frac{1}{p'}}\\
 &\leq&C
 \parallel u_0\parallel_{\dot{FB}^{4-2\alpha-\frac{3}{p}}_{p,q}}
 \left( J_1+J_2\right)^{\frac{1}{q'}},
\end{eqnarray*} 
where

\begin{equation*}
J_1=
 \displaystyle{\sum_{j\leq 0}}
\left(
\displaystyle{\int_{2^{j}\cal{C}}}
\mid\xi\mid^{(-4+2\alpha+\frac{3}{p})p'}
\mid \hat{\psi}(\xi)\mid^{p'}\,d\xi
\right)^{\frac{q'}{p'}}
\end{equation*}
and 

\begin{equation*}
J_2=
\displaystyle{\sum_{j> 0}}
\left(
\displaystyle{\int_{2^{j}\cal{C}}}
\mid\xi\mid^{(-4+2\alpha+\frac{3}{p})p'}
\mid \hat{\psi}(\xi)\mid^{p'}\,d\xi
\right)^{\frac{q'}{p'}}.
\end{equation*}
Let us take $a=4-2\alpha<3$ and $A_{1}=\frac{1}{p}(\alpha-\frac{1}{2})$. Thus, we can conclude that
\begin{eqnarray*}
J_1&\leq&
\displaystyle{\sum_{j\leq 0}}2^{jA_1q'}
\left(
\displaystyle{\int_{\mid\xi\mid\leq 2}}
\frac{1}{\mid\xi\mid^{a}}\mid\xi\mid^{A_{1}p'}\mid\hat{\psi}(\xi)\mid^{p'}\,d\xi\right)^{\frac{q'}{p'}}\\
&\leq&
\displaystyle{\sum_{j\leq 0}}2^{jA_1q'}
\left(
\displaystyle{\int_{\mid\xi\mid\leq 2}}
\frac{1}{\mid\xi\mid^{a}}\,d\xi
\right)^{\frac{q'}{p'}}\cdot
\left(
\displaystyle{\sup_{\xi\in\mathbb{R}^{3}}}
\mid \xi\mid^{A_1}\mid\hat{\psi}(\xi)\mid\,d\xi
\right)^{q'}<+\infty.
\end{eqnarray*}
Working similarly, we can take $b=3+p(\alpha-\frac{1}{2})>3$ and $A_2=(p+3)(\alpha-\frac{1}{2})$, in order to get the estimate

\begin{equation*}
J_2\leq
\displaystyle{\sum_{j> 0}}2^{-jA_2q'}
\left(
\displaystyle{\int_{\mid\xi\mid\geq 4}}
\frac{1}{\mid\xi\mid^{b}}\,d\xi
\right)^{\frac{q'}{p'}}\cdot
\left(
\displaystyle{\sup_{\xi\in\mathbb{R}^{3}}}
\mid \xi\mid^{2(p+2)(\alpha-\frac{1}{2})}\mid\hat{\psi}(\xi)\mid\,d\xi
\right)^{q'}<+\infty.
\end{equation*}
These last two estimatives implies that $I_1< +\infty$. 
For estimative of the term $I_2$, working as in estimatives for the term 
$I_1$, we observe that

\begin{equation*}
I_2\leq C
\parallel u_0\parallel_{\dot{FB}^{4-2\alpha-\frac{3}{p}}_{p,q}}
\left(
L_1+L_2
\right)^{\frac{1}{q'}},
\end{equation*}
where 

\begin{equation*}
L_1=
\displaystyle{\sum_{j\leq 0}}
\left(
\displaystyle{\int_{2^{j}\cal{C}}}
\mid\xi\mid^{(-4+4\alpha+\frac{3}{p})p'}
\mid \hat{\psi}(\xi)\mid^{p'}\,d\xi
\right)^{\frac{q'}{p'}}
\end{equation*}
and 

\begin{equation*}
L_2=
\displaystyle{\sum_{j> 0}}
\left(
\displaystyle{\int_{2^{j}\cal{C}}}
\mid\xi\mid^{(-4+4\alpha+\frac{3}{p})p'}
\mid \hat{\psi}(\xi)\mid^{p'}\,d\xi
\right)^{\frac{q'}{p'}}.
\end{equation*}
Taking $\frac{1}{2}<\alpha$, $\tilde{a}=\frac{5}{2}$ and 
$\tilde{A}_{1}=\frac{1}{4p}+2\alpha-\frac{3}{4}>0$, we get the following estimative

\begin{equation*}
L_1\leq
\displaystyle{\sum_{j\leq 0}}2^{j\tilde{A}_{1}q'}
\left(
\displaystyle{\int_{\mid\xi\mid\leq 2}}
\frac{1}{\mid\xi\mid^{\tilde{a}}}\,d\xi
\right)^{\frac{q'}{p'}}\cdot
\left(
\displaystyle{\sup_{\xi\in\mathbb{R}^{3}}}
\mid \xi\mid^{2(\alpha-\frac{1}{2})+\frac{1}{4}(1+\frac{1}{p}) }\mid\hat{\psi}(\xi)\mid\,d\xi
\right)^{q'}<+\infty.
\end{equation*}
Also, if we consider $\frac{1}{2}<\alpha$, $\tilde{b}=\frac{1}{2}(p+6)$ and $\tilde{A}_{2}=\frac{1}{2}(p+1)$, we get

\begin{equation*}
L_2\leq
\displaystyle{\sum_{j\leq 0}}2^{-j\tilde{A}_{2}q'}
\left(
\displaystyle{\int_{\mid\xi\mid\geq 4}}
\frac{1}{\mid\xi\mid^{\tilde{b}}}\,d\xi
\right)^{\frac{q'}{p'}}\cdot
\left(
\displaystyle{\sup_{\xi\in\mathbb{R}^{3}}}
\mid \xi\mid^{4(\alpha-\frac{1}{2}) + (p+1)}\mid\hat{\psi}(\xi)\mid\,d\xi
\right)^{q'}<+\infty.
\end{equation*}
These last two estimatives implies that $I_2< +\infty$. The inequality $(\ref{des4.2A})$  implies that

\begin{equation*}
\mid \langle S^{\alpha}_{\Omega}(t)u_0-u_0,\psi\rangle\mid\leq
C_{\Omega,\nu}t,\,\,\mbox{with}\,\,C_{\Omega,\nu}<+\infty.
\end{equation*}
Therefore, we get the desired convergence

\begin{equation*}
S^{\alpha}_{\Omega}(t)u_0\longrightarrow u_0\,\,\mbox{in}\,\,{\cal S}',\,\,
\,\mbox{as}\,\,t\longrightarrow 0^{+}.
\end{equation*}
Now we shall prove the following convergence

\begin{equation*}
S^{\alpha}_{\Omega}(t)u_0\longrightarrow S^{\alpha}_{\Omega}(s)u_0\,\,\mbox{in}\,\,{\cal S}',\,\,
\,\mbox{as}\,\,t\longrightarrow s\,\,\mbox{for each}\,\, s>0.
\end{equation*}
The Plancherel formula and the semigroup property give us
\begin{eqnarray*}
\mid \langle S^{\alpha}_{\Omega}(t)u_0-S^{\alpha}_{\Omega}(s)u_0,
\psi\rangle
&=&
\mid \langle S^{\alpha}_{\Omega}(s))^{\wedge}
[(S^{\alpha}_{\Omega})(t-s))^{\wedge}\hat{u}_0-\hat{u}_0],\hat{\psi}
\rangle\mid\\
&\leq& C
\left(
\displaystyle{\int_{\mathbb{R}^{3}}}
\mid 1-e^{-\nu (t-s)\mid\xi\mid^{2\alpha}}\cos(
\frac{\Omega\xi_3}{\mid\xi\mid}(t-s))
\mid\mid\hat{u}_0(\xi)\mid
\mid \hat{\psi}(\xi\mid\,d\xi)+\right.\\
& &
\left.
\displaystyle{\int_{\mathbb{R}^{3}}}
\mid e^{-\nu (t-s)\mid\xi\mid^{2\alpha}}\sin(
\frac{\Omega\xi_3}{\mid\xi\mid}(t-s))
\mid\mid\hat{u}_0(\xi)\mid
\mid \hat{\psi}(\xi\mid\,d\xi)\right)\\
&\leq& C_{\Omega,\nu}\mid t-s\mid.
\end{eqnarray*}
This estimative gives the desired weak continuity at $s>0$ and complete the proof.
\end{proof}	
Next lemma describes the gain of regularity for the integral of the semigroup applied to the projection of the external force. In order to get this gain of regularity, we use the integral of the Fourier transform of the semigroup $\left(S^{\alpha}_{
\Omega}(t) \right)^{\wedge}$.

\begin{lema}\label{lemmawt}
For each $F\in {\cal L}^{\infty}(I;\dot{FB}^{4-4\alpha-\frac{3}{p}}_{p,q})$ it is valid that

\begin{equation*}
w(t)=\displaystyle{\int_{0}^{t}}
S^{\alpha}_{\Omega}(t-\tau)\mathbb{P}F(\tau)\,
d\tau\in\dot{FB}^{4-2\alpha-\frac{3}{p}}_{p,q},\,\,\mbox{for each}\,\,t>0, 
\end{equation*}
the boundedness

\begin{equation*}
\parallel w\parallel_{{\cal X}}\leq \frac{C}{\nu}
\parallel F\parallel_{ {\cal L}^{\infty}(I;\dot{FB}^{4-4\alpha-\frac{3}{p}}_{p,q})},
\end{equation*}	
for some constant $C>0$ and the weak convergence $w(t)\rightharpoonup w(s)$ as $t\rightarrow s$, for each $s\geq 0$.
\end{lema}	
\begin{proof}
By def{}inition of the norm in $\dot{FB}^{4-2\alpha-\frac{3}{p}}_{p,q}$, looking at the expression for the semigroup $\left(S^{\alpha}_{
\Omega}(t) \right)^{\wedge}$ and the bound for the matrix $R(\xi)$ we have
\begin{eqnarray*}
\parallel w(t)\parallel_{\dot{FB}^{4-4\alpha-\frac{3}{p}}_{p,q}}&=&
\parallel 
2^{j(4-2\alpha-\frac{3}{p})}
\parallel\varphi_j\hat{w}(t)\parallel_{L^{p}}
\parallel_{l^{q}}\\
&\leq& C
\parallel
\displaystyle{\int_{0}^{t}}
2^{j(4-2\alpha-\frac{3}{p})}
e^{-\nu(t-\tau)2^{2\alpha(j-1)}}
\parallel \varphi_j\hat{F}(\tau,\cdot)\parallel_{
L^{p}}\,d\tau\parallel_{l^{q}}\\
&\leq&
C
\parallel
\displaystyle{\int_{0}^{t}}
2^{j(4-2\alpha-\frac{3}{p})}
e^{-\nu(t-\tau)2^{2\alpha(j-1)}}
\,d\tau
\parallel \varphi_j\hat{F}\parallel_{L^{\infty}(I;
	L^{p})}\parallel_{l^{q}}\\
&\leq&
\frac{C}{\nu} 
\parallel 2^{j(4-4\alpha-\frac{3}{p})}
\parallel \varphi_j\hat{F}\parallel_{L^{\infty}(I;
	L^{p})}\parallel_{l^{q}}\\
&\leq&
\frac{C}{\nu}
\parallel F\parallel_{{\cal L}^{\infty}(I;\dot{FB}^{4-4\alpha-\frac{3}{p}}_{p,q})}
\end{eqnarray*}	
Thus, $w(t)\in \dot{FB}^{4-2\alpha-\frac{3}{p}}_{p,q}$, for each $t>0$. Also we have the estimate
\begin{eqnarray*}
\parallel w\parallel_{
{\cal L}^{\infty}(I;\dot{FB}^{4-2\alpha-\frac{3}{p}}_{p,q})}&=&
\parallel 
2^{j(4-2\alpha-\frac{3}{p})}
\parallel\varphi_j\hat{w}(t)\parallel_{L^{\infty}(I;
	L^{p})}
\parallel_{l^{q}}\\
&\leq&\frac{C}{\nu}
\parallel F\parallel_{{\cal L}^{\infty}(I;\dot{FB}^{4-4\alpha-\frac{3}{p}}_{p,q})}.
\end{eqnarray*}
Let us show that $w(t)\rightharpoonup w(0)=0$ in ${\cal S}'$ as $t \rightarrow 0^{+}$. For each $\psi\in{\cal S}$, we have
\begin{eqnarray*}
\mid \langle w(t),\psi\rangle\mid
&=&
\mid 
\displaystyle{\int_{\mathbb{R}^{3}}}
\hat{w}(t)\hat{\psi}(\xi)\,d\xi\mid\\
&\leq&
\displaystyle{\int_{0}^{t}}
\displaystyle{\int_{\mathbb{R}^{3}}}
\displaystyle{\sum_{j\in\mathbb{Z}}}\,
e^{-\nu(t-\tau)2^{2\alpha(j-1)} }
\displaystyle{\int_{2^{j}{\cal C}}}
\mid \hat{\psi}(\xi)\mid
\mid\varphi_j(\xi)\mid
\mid \hat{F}(\tau,\xi)\mid\,d\xi\,d\tau\\
&\leq& 
\displaystyle{\sum_{j\in\mathbb{Z}} }\,
\displaystyle{\int_{0}^{t}}
e^{-\nu(t-\tau)2^{2\alpha(j-1)} }\,d\tau
	\left( 
\displaystyle{\int_{2^{j}{\cal C}}}
\mid \hat{\psi}(\xi)\mid^{p'}\,d\xi	
	\right)^{\frac{1}{p'}}
	\parallel\varphi_j\hat{F}\parallel_{L^{\infty}(I;L^{p})} \\
	&\leq&
t\,
\displaystyle{
\sup_{j\in\mathbb{Z}}
\mid\frac{1-e^{-\nu t2^{2(j-1)\alpha}}}{\nu t 2^{2(j-1)\alpha}}\mid
}	
\left[
\displaystyle{\sum_{j\in\mathbb{Z}} }\,
\left(
\displaystyle{\int_{2^{j}{\cal C}}}
\mid\xi\mid^{-p'(4-4\alpha-\frac{3}{p})}
\mid\hat{\psi}(\xi)\mid^{p'}\,d\xi
\right)^{\frac{q'}{p'}}\right]^{\frac{1}{q'}}\\
&\times& 
\parallel F\parallel_{{\cal L}^{\infty}(I;\dot{FB}^{4-4\alpha-\frac{3}{p}}_{p,q})}
\end{eqnarray*}
and thus 

\begin{equation*}
\mid \langle w(t),\psi\rangle\mid \leq C\cdot t 
\parallel F\parallel_{{\cal L}^{\infty}(I;\dot{FB}^{4-4\alpha-\frac{3}{p}}_{p,q})}.
\end{equation*}
This inequality implies the desired convergence.
Now, we shall prove that $w(t)\rightharpoonup w(s)$ in ${\cal S}'$ as 
$t \rightarrow s$. In fact, we can assume that $t<s$. For each $\psi
\in{\cal S}$, we have

\begin{equation*}
\langle w(t)-w(s),\psi\rangle=\langle \hat{w}(t)-\hat{w}(s),\hat{\psi} \rangle=J_1+J_2,
\end{equation*}
where

\begin{equation*}
J_1=
\displaystyle{\int_{\mathbb{R}^{3}}}
\displaystyle{\int_{0}^{t}}
\left[
(S^{\alpha}_{\Omega})^{\wedge}(t-\tau,\xi)-
(S^{\alpha}_{\Omega})^{\wedge}(s-\tau,\xi)
\right]
\hat{\mathbb{P}}(\xi)\hat{F}(\tau,\xi)\hat{\psi}(\xi)\,d\tau\,d\xi
\end{equation*}
and 

\begin{equation*}
J_2=
\displaystyle{\int_{\mathbb{R}^{3}}}
\displaystyle{\int_{s}^{t}}
(S^{\alpha}_{\Omega})^{\wedge}(s-\tau,\xi)
\hat{\mathbb{P}}(\xi)\hat{F}(\tau,\xi)\hat{\psi}(\xi)\,d\tau\,d\xi.
\end{equation*}
By the semigroup property, we have

\begin{equation*}
J_1=
\displaystyle{\int_{\mathbb{R}^{3}}}
\displaystyle{\int_{0}^{t}}
(S^{\alpha}_{\Omega})^{\wedge}(t-\tau,\xi)
(I- (S^{\alpha}_{\Omega})^{\wedge}(s-t,\xi) )
\hat{\mathbb{P}}(\xi)\hat{F}(\tau,\xi)\hat{\psi}(\xi)\,d\tau\,d\xi
\end{equation*}
and thus the estimates
\begin{eqnarray*}
\mid J_1\mid
&\leq&C
\left[
\displaystyle{\int_{\mathbb{R}^{3}}}
\displaystyle{\int_{0}^{t}}
\mid 1-e^{-\nu(s-t)\mid\xi \mid^{2\alpha}}\cos(\frac{\Omega\xi_3}{\mid\xi\mid}(s-t))\mid
\mid \hat{F}(\tau,\xi)\mid \mid\hat{\psi}(\xi)\mid\,d\tau\,d\xi \right.\\
&+&
\left.
\displaystyle{\int_{\mathbb{R}^{3}}}
\displaystyle{\int_{0}^{t}}
\mid e^{-\nu(s-t)\mid\xi \mid^{2\alpha}}\sin(\frac{\Omega\xi_3}{\mid\xi\mid}(s-t))\mid
\mid \hat{F}(\tau,\xi)\mid \mid\hat{\psi}(\xi)\mid\,d\tau\,d\xi
\right]\\
&\leq&
C\cdot\mid t-s\mid\cdot
\displaystyle{\int_{\mathbb{R}^{3}}}
\displaystyle{\int_{0}^{t}}
(
\Omega+\nu\mid\xi\mid^{2\alpha} )
\mid \hat{F}(\tau,\xi)\mid \mid\hat{\psi}(\xi)\mid\,d\tau\,d\xi
\\
&\leq&
C_{\Omega,\nu}\cdot\mid t-s\mid \cdot
\parallel F\parallel_{{\cal L}^{\infty}(I;\dot{FB}^{4-4\alpha-\frac{3}{p}}_{p,q})}.
\end{eqnarray*}
For the term $J_2$, we have 
\begin{eqnarray*}
\mid J_2\mid
&\leq&
\displaystyle{\int_{\mathbb{R}^{3}}}
\displaystyle{\int_{s}^{t}}
\mid \hat{F}(\tau,\xi)\mid \mid\hat{\psi}(\xi)\mid\,d\tau\,d\xi\\
&\leq&
\displaystyle{\int_{s}^{t}}
\displaystyle{\sum_{j\in\mathbb{Z}}}\,
\displaystyle{\int_{2^{j}\cal C}}
\mid \hat{\psi}(\xi)\mid \mid\varphi_j\mid \mid\hat{F}(\tau,\xi)\mid\,d\tau\\
&\leq&C\cdot
\mid t-s\mid \cdot
\parallel F\parallel_{{\cal L}^{\infty}(I;\dot{FB}^{4-4\alpha-\frac{3}{p}}_{p,q})}.
\end{eqnarray*}
Therefore 
\begin{eqnarray*}
\mid \langle w(t)-w(s),\psi\rangle\mid
 &\leq&
 \mid J_1\mid+\mid J_2\mid\\
 &\leq &
 C_{\Omega,\nu}\cdot\mid t-s\mid \cdot
 \parallel F\parallel_{{\cal L}^{\infty}(I;\dot{FB}^{4-4\alpha-\frac{3}{p}}_{p,q})}
\end{eqnarray*}
and thus we get the desired convergence.
\end{proof}	
Next lemma establish a estimative for the paraproduct of tempered distributions on the space ${\cal X}$. This is one of the fundamental tools in order to control the non-linear term for the system (FNSC).

\begin{lema}\label{productX}
For $\frac{1}{2}<\alpha<\frac{1}{4}\left( 5-\frac{3}{p}\right)$, we have the following estimative

\begin{equation*}
\parallel u\otimes v\parallel_{
	{\cal L}^{\infty}(I;\dot{FB}^{5-4\alpha-\frac{3}{p}}_{p,q})}
\leq \parallel u\parallel_{{\cal X}}
\parallel v\parallel_{{\cal X}}.	
\end{equation*}	
\end{lema}	
\begin{proof}
Using the paraproduct formula, we can consider the expression

\begin{equation*}
\Delta_j(u v)=
\displaystyle{\sum_{\mid k-j\mid\leq 4}}
\Delta_j(S_{k-1}u\Delta_k v)
+
\displaystyle{\sum_{\mid k-j\mid\leq 4}}
\Delta_j(S_{k-1}v\Delta_k u)
+\displaystyle{\sum_{k\geq j-2}}
\Delta_j
(\Delta_k u\tilde{\Delta}_k v)
\end{equation*}
which, applying the Fourier transform, implies
\begin{eqnarray*}
\varphi_j (uv)^{\wedge}
&=&
\displaystyle{\sum_{\mid k-j\mid\leq 4}}
\varphi_{j}((S_{k-1}u)^{\wedge}\ast \varphi_k \hat{v})
+
\displaystyle{\sum_{\mid k-j\mid\leq 4}}
\varphi_j((S_{k-1}v)^{\wedge}\ast\varphi_k \hat{u})
+\displaystyle{\sum_{k\geq j-2}}
\varphi_j
(\varphi_k u\ast \tilde{\varphi}_k v)\\
&=& I_j+II_j+III_j.
\end{eqnarray*} 
\newline
In order to get estimatives for $I_j$, $II_j$ and $III_j$, we use H\"older inequality for series and the bound inequality $\parallel\varphi_{k'}\hat{u}\parallel_{L^{1}}\leq 2^{k'(3-\frac{3}{p})}\parallel \varphi_{k'}\hat{u}\parallel_{L^{p}}$. Let us start with the estimate for the term $I_j$ as follows,
\begin{eqnarray*}
\parallel I_j \parallel_{L^{\infty}(I;L^{p})} 
&\leq&
\displaystyle{\sum_{\mid k-j\mid\leq 4}}
\parallel (S_{k-1}u)^{\wedge}\ast \varphi_k\hat{v}\parallel_{L^{\infty}(I;L^{p})}\\
&\leq&
\displaystyle{\sum_{\mid k-j\mid\leq 4}}\,
\displaystyle{\sum_{k'<k-1}}
\parallel \varphi_{k'}\hat{u}\parallel_{L^{\infty}(I;L^{1})}
\parallel \varphi_{k}\hat{v}\parallel_{L^{\infty}(I;L^{p})}\\
&\leq &
\displaystyle{\sum_{\mid k-j\mid\leq 4}}\,
\displaystyle{\sum_{k'<k-1}}
2^{k'(3-\frac{3}{p})}
\parallel 
\varphi_{k'}\hat{u}\parallel_{L^{\infty}(I;L^{p})}
\parallel \varphi_{k}\hat{v}\parallel_{L^{\infty}(I;L^{p})}\\
&\leq&
\displaystyle{\sum_{\mid k-j\mid\leq 4}}\,
\displaystyle{\sum_{k'<k-1}}
2^{k'(-1+2\alpha)} 2^{k'(4-2\alpha-\frac{3}{p})}
\parallel 
\varphi_{k'}\hat{u}\parallel_{L^{\infty}(I;L^{p})}
\parallel \varphi_{k}\hat{v}\parallel_{L^{\infty}(I;L^{p})}\\
&\leq&
\parallel u\parallel_{{\cal L}^{\infty}(I;\dot{FB}^{4-2\alpha-\frac{3}{p}}_{p,q})}
\displaystyle{\sum_{\mid k-j\mid\leq 4}}
2^{k(-1+2\alpha)}
\parallel \varphi_k\hat{v}\parallel_{L^{\infty}(I;L^{p})}
\\
&\leq&
\parallel u\parallel_{{\cal X}}
\displaystyle{\sum_{\mid k-j\mid\leq 4}}
2^{k(-5+4\alpha+\frac{3}{p})}
2^{k(4-2\alpha-\frac{3}{p})}
\parallel \varphi_k\hat{v}\parallel_{L^{\infty}(I;L^{p})} \\
&\leq&
\parallel u\parallel_{{\cal X}}
2^{j(-5+4\alpha+\frac{3}{p})}
\displaystyle{\sum_{k}}
2^{-(j-k)(-5+4\alpha+\frac{3}{p})}\chi_{\{l;\mid l\mid\leq 4 \}}(j-k)
2^{k(4-2\alpha-\frac{3}{p})}
\parallel \varphi_k\hat{v}\parallel_{L^{\infty}(I;L^{p})}
\\
&\leq&
\parallel u\parallel_{{\cal X}}
2^{j(-5+4\alpha+\frac{3}{p})}
\displaystyle{\sum_{k}}a_{j-k}
b_k\\
&=&\parallel u\parallel_{{\cal X}}
2^{j(-5+4\alpha+\frac{3}{p})}
(a_l\ast b_k)_{j},
\end{eqnarray*}
where $a_l=2^{-l(-5+4\alpha+\frac{3}{p})}\chi_{\{l;\mid l\mid\leq 4\}}(l)$, $b_{k}=2^{k(4-2\alpha-\frac{3}{p})}\parallel\varphi_k\hat{v}\parallel_{L^{\infty}(I;L^{p})}$ and we use the assumption $\alpha>\frac{1}{2}$ and thus, multiplying by $2^{j(5-4\alpha-\frac{3}{p})}$ on both the sides of the above estimate and then applying Young's inequality for series, we get the following estimate,

\begin{equation}\label{Ij}
\parallel 2^{j(5-4\alpha-\frac{3}{p})}
\parallel I_j\parallel_{L^{\infty}(I;L^{p})}
\parallel_{l^{q}}\leq
C 
\parallel u\parallel_{{\cal X}}
\parallel v\parallel_{{\cal X}}.
\end{equation}
The term $J_2$ we estimate in a similar way as the term $J_1$ and we get

\begin{equation}\label{IIj}
\parallel 2^{j(5-4\alpha-\frac{3}{p})}
\parallel II_j\parallel_{L^{\infty}(I;L^{p})}
\parallel_{l^{q}}\leq
C 
\parallel u\parallel_{{\cal X}}
\parallel v\parallel_{{\cal X}}.
\end{equation}
Let us estimate the last term $I_3$
\begin{eqnarray*}
\parallel
III_j\parallel_{L^{\infty}(I;L^{p})}
&\leq&
\displaystyle{\sum_{k\geq j-2}}
\parallel
\varphi_j(
\varphi_k \hat{u}\ast 
\tilde{\varphi}_k\hat{v}
)
\parallel_{L^{\infty}(I;L^{p})}\\
&\leq &
\displaystyle{\sum_{k\geq j-2}}
\parallel \varphi_k\hat{u}\parallel_{L^{\infty}(I;L^{1})}
\parallel \tilde{\varphi}_k\hat{v}\parallel_{L^{\infty}(I;L^{p})}\\
&\leq&
\displaystyle{\sum_{k\geq j-2}}
2^{k(3-\frac{3}{p})}
\parallel \varphi_k\hat{u}\parallel_{L^{\infty}(I;L^{p})}
\displaystyle{\sum_{\mid k'-k\mid\leq 1}}\,
2^{-k'(4-2\alpha-\frac{3}{p})}
2^{k'(4-2\alpha-\frac{3}{p})} 
\parallel \varphi_{k'}\hat{v}\parallel_{L^{\infty}(I;L^{p})}\\
&\leq &
\parallel v\parallel_{{\cal X}}
\displaystyle{\sum_{k\geq j-2}}
2^{k(3-\frac{3}{p})}
\parallel \varphi_k\hat{u}\parallel_{L^{\infty}(I;L^{p})} 
2^{-k(4-2\alpha-\frac{3}{p})}\\
&\leq&
\parallel v\parallel_{{\cal X}}
\displaystyle{\sum_{k\geq j-2}}
2^{k(-1+2\alpha)} 
\parallel \varphi_k\hat{u}
\parallel_{L^{\infty}(I;L^{p})}\\
&=&
\parallel v\parallel_{{\cal X}}
2^{j(-5+4\alpha+\frac{3}{p})}
\displaystyle{\sum_{k\geq j-2}}
2^{-(j-k)(-5+4\alpha+\frac{3}{p})}
2^{k(4-2\alpha-\frac{3}{p})}
\parallel \varphi_k\hat{u}
\parallel_{L^{\infty}(I;L^{p})}\\
&=&
\parallel v\parallel_{{\cal X}}
2^{j(-5+4\alpha+\frac{3}{p})}
\displaystyle{\sum_{k}}
2^{-(j-k)(-5+4\alpha+\frac{3}{p})}
\chi_{\{l;l \leq 2\}}(j-k) b_k\\
&=& 
\parallel v\parallel_{{\cal X}}
(a_l\ast b_k)_j,
\end{eqnarray*} 
where $a_{l}=2^{-l(-5+4\alpha+\frac{3}{p})}\chi_{\{l;l\leq 2\}}(l)$ and $b_{k}=2^{k(4-2\alpha-\frac{3}{p})}\parallel\varphi_{k}\hat{u}\parallel_{L^{\infty}(I;L^{p})}$. Thus, multiplying by $2^{j(5-4\alpha-\frac{3}{p})}$ on both the sides of the above estimate and then applying Young's inequality for series, we get the following estimate

\begin{equation}\label{IIIj}
\parallel 2^{j(5-4\alpha-\frac{3}{p})}
\parallel I_j\parallel_{L^{\infty}(I;L^{p})}
\parallel_{l^{q}}\leq
C 
\parallel u\parallel_{{\cal X}}
\parallel v\parallel_{{\cal X}},
\end{equation}
where we use the assumption $\alpha<\frac{1}{4}(5-\frac{3}{p})$.
Finally, estimates $(\ref{Ij})$, $(\ref{IIj})$ and $(\ref{IIIj})$ implies the desired estimate for the paraproduct.
\end{proof}	 
The following proposition give us the continuity of the bilinear 
operator $B$ in the space ${\cal X}={\cal L}^{\infty}(I; \dot{FB}^{4-2\alpha-\frac{3}{p}}_{p,q})$.

\begin{pp}\label{constantK}
Let be $\frac{1}{2}<\alpha<\frac{1}{4}\left(5-\frac{3}{p}\right)$. The bilinear operator $B$ defined in (\ref{bilinearterm}) is continuous in ${\cal X}$. Moreover, there is a constant $K>0$ independent of the Coriolis parameter $\Omega$ such that

\begin{equation*}
\parallel B(u,v)\parallel_{{\cal X}}\leq 
K\parallel u\parallel_{{\cal X}}
\parallel v\parallel_{{\cal X}}.
\end{equation*}
\end{pp}	
\begin{proof}
In fact, recalling the definition of the bilinear operator $B$ and the norm of the space ${\cal X}$, we have
\begin{eqnarray*}
\parallel B(u,v)\parallel_{{\cal X}}
&=&
\parallel 2^{j(4-2\alpha-\frac{3}{p})}
\parallel\varphi_j(B(u,v))^{\wedge}\parallel_{L^{\infty}(I;L^{p})}
\parallel_{l^{q}}\\
&=&
\parallel 2^{j(4-2\alpha-\frac{3}{p})}
\parallel \varphi_j 
\displaystyle{\int_{0}^{t}}
\left(
e^{-\nu(t-\tau)\mid\xi\mid^{2\alpha}}\cos(\frac{\Omega\xi_3}{\mid\xi\mid}(t-\tau))I+\right.\\
& &\left.
e^{-\nu(t-\tau)\mid\xi\mid^{2\alpha}}\sin(\frac{\Omega\xi_3}{\mid\xi\mid}(t-\tau))R(\xi)
\right)
 \xi\cdot
(u\otimes v)^{\wedge}(\tau,\xi)
\parallel_{L^{\infty}(I;L^{p})}
\parallel_{l^{q}}\\
&\leq&
C\parallel 2^{j(5-2\alpha-\frac{3}{p})}
\parallel
\displaystyle{\int_{0}^{t}}
e^{-\nu(t-\tau)2^{2\alpha(j-1)}}
\parallel \varphi_j
(u\otimes v)^{\wedge}(\tau,\xi)
\parallel_{L^{p}}\,d\tau
\parallel_{L^{\infty}(I)}
\parallel_{l^{q}}\\
&\leq&
C
\parallel 2^{j(5-4\alpha-\frac{3}{p})}
\parallel \varphi_j
(u\otimes v)^{\wedge}
\parallel_{L^{\infty}(I;L^{p})}
\parallel_{l^{q}}\\
&=& C
\parallel u\otimes v\parallel_{
{\cal L}^{\infty}(I;\dot{FB}^{5-4\alpha-\frac{3}{p}}_{p,q})
}\\
&\leq &
K
\parallel u\parallel_{{\cal X}}
\parallel v\parallel_{{\cal X}},
\end{eqnarray*}	
where we are using the assumption $\frac{1}{2}<\alpha<\frac{1}{4}\left(5-\frac{3}{p}\right)$ and the 
estimative \ref{productX} of the paraproduct in ${\cal X}$.
\end{proof}	
\begin{obs}
Lemma $\ref{lemmasemFC}$, Lemma $\ref{lemmawt}$, Proposition $\ref{constantK}$ and the fixed point lemma (Lemma \ref{fixedpoint}) permit us to prove Theorem \ref{teorema1}.
\end{obs}
In the next lemma we describe the asymptotic behaviour of the integral of the semigroup, for any external force which is small enough for large time.
\begin{lema}\label{lemma5.1}
Let $F$ be an external force such that $F\in{\cal L}^{\infty}(I;\dot{FB}^{4-4\alpha-\frac{3}{p}}_{p,q})$.

\begin{equation*}
\mbox{If}\,\,
\displaystyle{\lim_{t\rightarrow \infty}}
\parallel F(t)\parallel_{\dot{FB}^{4-4\alpha-\frac{3}{p}}_{p,q}}=0\,\,
\mbox{then}\,\,
\displaystyle{\lim_{t\rightarrow \infty}}
\parallel\displaystyle{\int_{0}^{t}}
S^{\alpha}_{\Omega}(t-\tau)\mathbb{P}F(\tau)\,d\tau
\parallel_{\dot{FB}^{4-2\alpha-\frac{3}{p}}_{p,q}}=0.
\end{equation*}	
\end{lema}	
\begin{proof}
By def{}inition of the norm in $\dot{FB}^{4-2\alpha-\frac{3}{p}}_{p,q}$ 
and the bound of the symbols of the Leray-projector $\mathbb{P}$, that is, 
$\parallel \hat{\mathbb{P}}\parallel\leq 2$ for each $\xi\in \mathbb{R}^{3}$, we have
\begin{eqnarray*}
\parallel\displaystyle{\int_{0}^{t}}
S^{\alpha}_{\Omega}(t-\tau)\mathbb{P}F(\tau)\,d\tau
\parallel_{\dot{FB}^{4-2\alpha-\frac{3}{p}}_{p,q}}
&=&
\parallel 2^{j(4-2\alpha-\frac{3}{p})}
\parallel\varphi_j 
\displaystyle{\int_{0}^{t}}
(S^{\alpha}_{\Omega})^{\wedge}(t-\tau,\xi)
\hat{\mathbb{P}}\hat{F}(\tau,\xi)\,d\tau
\parallel_{l^{q}}\\
&\leq&
\parallel 2^{j(4-2\alpha-\frac{3}{p})}
\parallel
\displaystyle{\int_{0}^{t}}
e^{-\nu(t-\tau)2^{2\alpha(j-1)}}
\parallel\varphi_j\hat{F}(\tau,\cdot)\parallel_{L^{p}}
\,d\tau
\parallel_{l^{q}}\\
&\leq&
J_1(t)+J_2(t),
\end{eqnarray*}
where 

\begin{equation*}
J_1(t)=
\parallel 2^{j(4-2\alpha-\frac{3}{p})}
\parallel
\displaystyle{\int_{0}^{t/2}}
e^{-\nu(t-\tau)2^{2\alpha(j-1)}}
\parallel\varphi_j\hat{F}(\tau,\cdot)\parallel_{L^{p}}
\,d\tau
\parallel_{l^{q}}
\end{equation*}
and 

\begin{equation*}
J_2(t)=
\parallel 2^{j(4-2\alpha-\frac{3}{p})}
\parallel
\displaystyle{\int_{t/2}^{t}}
e^{-\nu(t-\tau)2^{2\alpha(j-1)}}
\parallel\varphi_j\hat{F}(\tau,\cdot)\parallel_{L^{p}}
\,d\tau
\parallel_{l^{q}}
\end{equation*}
For the term $J_1(t)$ we estimate,
\begin{eqnarray*}
J_1(t)&\leq& C
\parallel
\displaystyle{\int_{0}^{t/2}}
\frac{2^{2\alpha}}{\nu(t-\tau)}\frac{\nu}{2^{2\alpha}}(t-\tau)2^{2\alpha j}
e^{-\frac{\nu}{2^{2\alpha}}(t-\tau)2^{2\alpha j}}
 2^{j(4-4\alpha-\frac{3}{p})}
\parallel\varphi_j\hat{F}(\tau,\cdot)\parallel_{L^{p}}
\,d\tau
\parallel_{l^{q}}\\
&\leq& \frac{C}{\nu}\,\left(
\displaystyle{\sup_{r\geq 0}}\,re^{-r}\right)\cdot
\displaystyle{\int_{0}^{t/2}}\frac{1}{t-\tau}
\parallel F(\tau)\parallel_{\dot{FB}^{4-4\alpha-\frac{3}{p}}_{p,q}}\,d\tau\\
&\leq&
\frac{C}{\nu}\displaystyle{\int_{0}^{1/2}}\frac{1}{1-s}
\parallel F(ts)\parallel_{\dot{FB}^{4-4\alpha-\frac{3}{p}}_{p,q}}\,ds,
\end{eqnarray*}
where in the last inequality we use the change of variables $\tau=ts$. The assumption on $F$ and the Lebesgue dominated convergence theorem give us 

\begin{equation}\label{EE1}
\displaystyle{\lim_{t\rightarrow\infty}}J_1(t)=0.
\end{equation}
For the  term $J_2(t)$ we proceed as follows
\begin{eqnarray*}
J_2(t)&\leq& 
\parallel
\displaystyle{\int_{t/2}^{t}}
2^{2\alpha j}e^{-\nu(t-\tau)2^{2(j-1)}}\,d\tau\,
2^{j(4-4\alpha-\frac{3}{p})}
\parallel\varphi_j\hat{F}(\tau)
\parallel_{{\cal L}^{\infty}(t/2\leq \tau\leq t;L^{p})}
\parallel_{l^{q}}\\
&\leq&
\frac{C}{\nu}
\parallel F(\tau)
\parallel_{{\cal L}^{\infty}(t/2\leq\tau\leq t;\dot{FB}^{4-4\alpha-\frac{3}{p}}_{p,q})}
\end{eqnarray*}
and thus the assumption on $F$ give us

\begin{equation}\label{EE2}
\displaystyle{\lim_{t\rightarrow\infty}}J_2(t)=0.
\end{equation}
Therefore, $(\ref{EE1})$ and $(\ref{EE2})$ give us the required conclusion.
\end{proof}	
Now we shall prove Theorem \ref{teorema2}:
\begin{proof}
By Theorem \ref{teorema1} both the solutions $u$ and $v$ satisfy

\begin{equation*}
\parallel u\parallel_{
{\cal L}^{\infty}(I;\dot{FB}^{4-2\alpha-\frac{3}{p}}_{p,q})
}\leq 2\varepsilon <\frac{1}{2K}\,\,\mbox{and}\,\,
\parallel v\parallel_{
	{\cal L}^{\infty}(I;\dot{FB}^{4-2\alpha-\frac{3}{p}}_{p,q})
}\leq 2\varepsilon <\frac{1}{2K}
\end{equation*}
Since $u$ and $v$ are mild solutions and 

\begin{equation*}
\parallel w_1\otimes w_2\parallel_{\dot{FB}^{5-4\alpha-\frac{3}{p}}_{p,q} }
\leq C
\parallel w_1\parallel_{\dot{FB}^{4-2\alpha-\frac{3}{p}}_{p,q}}
\parallel w_2\parallel_{\dot{FB}^{4-2\alpha-\frac{3}{p}}_{p,q}},
\end{equation*}	
we have the estimative

\begin{equation}\label{F1barra}
\parallel u(t)-v(t)\parallel_{\dot{FB}^{4-2\alpha-\frac{3}{p}}_{p,q}}
\leq g(t)+I_1(t)+I_2(t),
\end{equation}
where 

\begin{equation*}
g(t)=\parallel S^{\alpha}_{\Omega}(t)(u_0-v_0)\parallel_{\dot{FB}^{4-2\alpha-\frac{3}{p}}_{p,q}} + 
\parallel
\displaystyle{\int_{0}^{t}}
S^{\alpha}_{\Omega}(t-\tau)\mathbb{P}(F(\tau)-G(\tau))\,d\tau 
\parallel_{
\dot{FB}^{4-2\alpha-\frac{3}{p}}_{p,q},
}
\end{equation*}

\begin{eqnarray*}
I_1(t)&=&C_1K
\displaystyle{\sup_{\xi\in\mathbb{R}^{3}}}
\displaystyle{\int_{0}^{\delta t}}
\mid\xi\mid^{2\alpha}e^{-\nu(t-\tau)\mid\xi\mid^{2\alpha}}
\parallel u(\tau)-v(\tau)\parallel_{\dot{FB}^{4-2\alpha-\frac{3}{p}}_{p,q}}
\,d\tau\times\\
& & 
\left(
\parallel u\parallel_{{\cal L}^{\infty}(I;\dot{FB}^{4-2\alpha-\frac{3}{p}}_{p,q})}+
\parallel v\parallel_{{\cal L}^{\infty}(I;\dot{FB}^{4-2\alpha-\frac{3}{p}}_{p,q})}
\right)\,\,\mbox{and}
\end{eqnarray*}
\begin{eqnarray*}
I_2(t)&=&C_2K
\displaystyle{\sup_{\xi\in\mathbb{R}^{3}}}
\displaystyle{\int_{\delta t}^{t}}
\mid\xi\mid^{2\alpha}e^{-\nu(t-\tau)\mid\xi\mid^{2\alpha}}
\parallel u(\tau)-v(\tau)\parallel_{\dot{FB}^{4-2\alpha-\frac{3}{p}}_{p,q}}
\,d\tau\times\\
& & 
\left(
\parallel u\parallel_{{\cal L}^{\infty}(I;\dot{FB}^{4-2\alpha-\frac{3}{p}}_{p,q})}+
\parallel v\parallel_{{\cal L}^{\infty}(I;\dot{FB}^{4-2\alpha-\frac{3}{p}}_{p,q})}
\right).
\end{eqnarray*}
Here the constant $0<\delta<1$ will be choosen later. \newline
Let us estimate the term $I_1(t)$: Using the change of variable $\tau=ts$ 
and the fact that 

\begin{equation*}
\displaystyle{\sup_{\xi\in\mathbb{R}^{3}}}
\mid\xi\mid^{2\alpha}e^{-(1-s)t\mid\xi\mid^{2\alpha}}=
(\nu(1-s)t)^{-1}
\displaystyle{\sup_{w\in\mathbb{R}^{3}}}
\mid w\mid^{2\alpha}e^{-\mid w\mid^{2\alpha}}=
(\nu(1-s)t)^{-1}e^{-1},
\end{equation*}
we get

\begin{equation}\label{F2}
I_1(t)\leq 4\varepsilon C_1e^{-1}K
\displaystyle{\int_{0}^{\delta}}
(1-s)^{-1}\parallel u(ts)-v(ts)\parallel_{
\dot{FB}^{4-2\alpha-\frac{3}{p}}_{p,q}}\,ds.
\end{equation}
For the term $I_2(t)$, since 

\begin{equation*}
\displaystyle{\sup_{\xi\in\mathbb{R}^{3}}}
\displaystyle{\int_{\delta t}^{t}}
\mid\xi\mid^{2\alpha}e^{-\nu(t-\tau)\mid\xi\mid^{2\alpha}}\,d\tau
=\frac{1}{\nu}\displaystyle{\sup_{\xi\in\mathbb{R}^{3}}}
(1-e^{-\nu t(1-\delta)\mid\xi\mid^{2\alpha}})=\frac{1}{\nu},
\end{equation*}
we get

\begin{equation}\label{F3}
J_2(t)\leq 4\varepsilon C_2 K
\parallel u-v\parallel_{{\cal L}^{\infty}(\delta t\leq \tau\leq t;
	\dot{FB}^{4-2\alpha-\frac{3}{p}}_{p,q})}
\end{equation}
Since $(u_0,v_0)\in \dot{FB}^{4-2\alpha-\frac{3}{p}}_{p,q}$ and 
$(F,G)\in{\cal L}^{\infty}(I; \dot{FB}^{4-4\alpha-\frac{3}{p}}_{p,q})$, we can see 
that $g\in L^{\infty}(I)$.
Assumption $(\ref{F1})$ and Lemma $\ref{lemma5.1}$ implies that

\begin{equation}\label{F4}
\displaystyle{\lim_{t\rightarrow+\infty}} g(t)=0.
\end{equation}
By using estimates $(\ref{F2})$ and $(\ref{F3})$ in $(\ref{F1barra})$ we have 

\begin{equation*}
 \parallel u(t)-v(t)\parallel_{\dot{FB}^{4-2\alpha-\frac{3}{p}}_{p,q}}
 \leq
g(t)+4\varepsilon Ke^{-1}C_1
\displaystyle{\int_{0}^{\delta}}
(1-s)^{-1}
\parallel u(ts)-v(ts)\parallel_{\dot{FB}^{4-2\alpha-\frac{3}{p}}_{p,q}}\,ds
\end{equation*}
\begin{eqnarray}\label{F5}
&+&
  4\varepsilon K C_2\parallel u-v\parallel_{{\cal L}^{\infty}(\delta t\leq \tau\leq t;\dot{FB}^{4-2\alpha-\frac{3}{p}}_{p,q})},
\end{eqnarray}
for all $t>0$.
Let us denote

\begin{equation}\label{F}
A=
\displaystyle{\limsup_{t\rightarrow \infty}}
\parallel u(t)-v(t)\parallel_{\dot{FB}^{4-2\alpha-\frac{3}{p}}_{p,q}}
=\displaystyle{\lim_{k\in\mathbb{N};k\rightarrow \infty}}
\displaystyle{\sup_{t\geq k}}
\parallel u(t)-v(t)\parallel_{\dot{FB}^{4-2\alpha-\frac{3}{p}}_{p,q}}.
\end{equation}
Since $(u,v)\in {\cal L}^{\infty}(I;\dot{FB}^{4-2\alpha-\frac{3}{p}}_{p,q})$ we have that 
$0\leq A\leq +\infty$. 
Let us show that in fact $A=0$. Lebesgue dominated convergence theorem gives 
\begin{eqnarray*}
&&\displaystyle{\limsup_{t\rightarrow \infty}}
\displaystyle{\int_{0}^{\delta}}(1-s)^{-1} 
\parallel u(ts)-v(ts)\parallel_{\dot{FB}^{4-2\alpha-\frac{3}{p}}_{p,q}}\,ds\\
&\leq&
\displaystyle{\lim_{k\in\mathbb{N};k\rightarrow \infty}}
\displaystyle{\int_{0}^{\delta}}(1-s)^{-1} 
\displaystyle{\sup_{t\geq k}}
\parallel u(ts)-v(ts)\parallel_{\dot{FB}^{4-2\alpha-\frac{3}{p}}_{p,q}}\,ds\\
&\leq&
A
\displaystyle{\int_{0}^{\delta}}(1-s)^{-1}\,ds
\end{eqnarray*}
and therefore 

\begin{equation}\label{F6}
\displaystyle{\limsup_{t\rightarrow \infty}}
\displaystyle{\int_{0}^{\delta}}(1-s)^{-1} 
\parallel u(ts)-v(ts)\parallel_{\dot{FB}^{4-2\alpha-\frac{3}{p}}_{p,q}}\,ds\leq 
A\log(\frac{1}{1-\delta})
\end{equation}
Also we have

\begin{equation*}
\displaystyle{\sup_{t\geq k}}
\parallel u-v\parallel_{{\cal L}^{\infty}(\delta t\leq \tau\leq t; \dot{FB}^{4-2\alpha-\frac{3}{p}}_{p,q})}\leq 
\parallel u-v\parallel_{{\cal L}^{\infty}(\delta k\leq \tau\leq t; \dot{FB}^{4-2\alpha-\frac{3}{p}}_{p,q})}
\end{equation*}
and thus

\begin{equation}\label{F7}
\displaystyle{\limsup_{t\rightarrow \infty}}
\parallel u-v\parallel_{{\cal L}^{\infty}(\delta t\leq \tau\leq t; \dot{FB}^{4-2\alpha-\frac{3}{p}}_{p,q})}\leq A.
\end{equation}
By the application of $\displaystyle{\limsup_{t\rightarrow \infty}}$ in $(\ref{F5})$ 
and using $(\ref{F})$, $(\ref{F6})$ and $(\ref{F7})$ we get
\begin{eqnarray*}
A&\leq&
 4\varepsilon Ke^{-1}C_1A\log(\frac{1}{1-\delta})+4\varepsilon KC_2A\\
&\leq& 
4\varepsilon KC\left( e^{-1}\log(\frac{1}{1-\delta})+1\right)A.
\end{eqnarray*}
If we reconsider $0<\varepsilon<\min\{\frac{1}{4K},\frac{1}{8KC}\}$, for $C=\max\{C_{1},C_{2}\}$, we can take $\delta$ small enough such that we get $A<A$. This implies $A=0$ and therefore 

 \begin{equation*}
 \displaystyle{\limsup_{t\rightarrow \infty}}
 \parallel u(t)-v(t)\parallel_{\dot{FB}^{4-2\alpha-\frac{3}{p}}_{p,q} }=0,
 \end{equation*}
 the required conclusion.
\end{proof}	 
Next corollary establish a necessary condition for the asympotic stability in time of close global mild solutions for the system (FNSC).
\begin{coro}\label{coroteorema2}
Let be $(u,v)\in {\cal L}^{\infty}(I;\dot{FB}^{4-2\alpha-\frac{3}{p}}_{p,q})$ a pair 
of solutions for the (FNSC)-system associated to the pair of initial data $(u_0,v_0)\in \dot{FB}^{4-2\alpha-\frac{3}{p}}_{p,q}$ and a pair of external forces	$(F,G)\in {\cal L}^{\infty}(I;\dot{FB}^{4-4\alpha-\frac{3}{p}}_{p,q})$.
If

\begin{equation*}
\displaystyle{\lim_{t\rightarrow\infty}}
\parallel u(t)-v(t)\parallel_{\dot{FB}^{4-2\alpha-\frac{3}{p}}_{p,q}}=0
\end{equation*}
then

\begin{equation}\label{F8}
\displaystyle{\lim_{t\rightarrow\infty}}
\parallel S^{\alpha}_{\Omega}(t)(u_0-v_0)+
\displaystyle{\int_{0}^{t}}
S^{\alpha}_{\Omega}(t-\tau)
\mathbb{P}(F(\tau)-G(\tau))\,d\tau
\parallel_{
\dot{FB}^{4-2\alpha-\frac{3}{p}}_{p,q}}=0.
\end{equation}
\end{coro}	
\begin{proof}
Since $u$ and $v$ are mild solutions for the (SNSC)-system, we have	
\begin{eqnarray*}
&\displaystyle{\lim_{t\rightarrow\infty}}
\parallel S^{\alpha}_{\Omega}(t)(u_0-v_0)+
\displaystyle{\int_{0}^{t}}
S^{\alpha}_{\Omega}(t-\tau)
\mathbb{P}(F(\tau)-G(\tau))\,d\tau
\parallel_{
	\dot{FB}^{4-2\alpha-\frac{3}{p}}_{p,q}}\leq\\
& \displaystyle{\lim_{t\rightarrow\infty}}
\parallel u(t)-v(t)\parallel_{\dot{FB}^{4-2\alpha-\frac{3}{p}}_{p,q}}
+C_1K
\displaystyle{\sup_{\xi\in\mathbb{R}^{3}}}
\displaystyle{\int_{0}^{\delta t}}
\mid\xi\mid ^{2\alpha} e^{-\nu(t-\tau)\mid\xi\mid^{2\alpha}}
\parallel u(\tau)-v(\tau)\parallel_{\dot{FB}^{4-2\alpha-\frac{3}{p}}_{p,q}}\,d\tau\\
&\times\left(
\parallel u\parallel_{{\cal L}^{\infty}(I;\dot{FB}^{4-2\alpha-\frac{3}{p}}_{p,q})}+
\parallel v\parallel_{{\cal L}^{\infty}(I;\dot{FB}^{4-2\alpha-\frac{3}{p}}_{p,q})}
\right) +\\
&C_2K
\displaystyle{\sup_{\xi\in\mathbb{R}^{3}}}
\displaystyle{\int_{\delta t}^{t}}
\mid\xi\mid ^{2\alpha} e^{-\nu(t-\tau)\mid\xi\mid^{2\alpha}}
\parallel u(\tau)-v(\tau)\parallel_{\dot{FB}^{4-2\alpha-\frac{3}{p}}_{p,q}}\,d\tau\\
&\times\left(
\parallel u\parallel_{{\cal L}^{\infty}(I;\dot{FB}^{4-2\alpha-\frac{3}{p}}_{p,q})}+
\parallel v\parallel_{{\cal L}^{\infty}(I;\dot{FB}^{4-2\alpha-\frac{3}{p}}_{p,q})}
\right),
\end{eqnarray*}
for some $\delta>0$ choosen later. Since $A=0$ and 
$(u,v)\in {\cal L}^{\infty}(I;\dot{FB}^{4-2\alpha-\frac{3}{p}}_{p,q})$ 
we can proceed as in Theorem \ref{teorema2} in order to get $(\ref{F8})$.
\end{proof}	
\begin{obs}
When $F=0$ the trivial solution is assymptotically stable, in the following sense: If we consider the constant $\varepsilon$ independent of the Coriolis parameter, provided by Theorem \ref{teorema1}, then Theorem \ref{teorema2} and Corollary \ref{coroteorema2} tell us that 
for the initial data $u_0\in\dot{FB}^{4-2\alpha-\frac{3}{p}}_{p,q}$ such that

\begin{equation*}
\parallel u_0\parallel_{\dot{FB}^{4-2\alpha-\frac{3}{p}}_{p,q}}<\varepsilon,
\end{equation*}
the associated solution goes to zero in $\dot{FB}^{4-2\alpha-\frac{3}{p}}_{p,q}$ if and only if 

\begin{equation*}
\displaystyle{\lim_{t\rightarrow 0} }
\parallel S^{\alpha}_{\Omega}(t)u_0
\parallel_{\dot{FB}^{4-2\alpha-\frac{3}{p}}_{p,q}}=0.
\end{equation*}
\end{obs}
Now, we shall prove Theorem \ref{teorema3}:
\begin{proof}
By Lemma \ref{equivStationary}, $u$ is a stationary solution if and only if 

\begin{equation*}
u=-
\displaystyle{\int_{0}^{\infty}}
S^{\alpha}_{\Omega}(\tau)\mathbb{P}\nabla\cdot (u\otimes u)\,d\tau +
\displaystyle{\int_{0}^{\infty}}
S^{\alpha}_{\Omega}(\tau)\mathbb{P}F\,d\tau.
\end{equation*}	
Thus, since 

\begin{equation*}
\displaystyle{\int_{0}^{\infty}}
(S^{\alpha}_{\Omega})^{\wedge}(\tau)\,d\tau=
\frac{1}{\nu^{2}\mid\xi\mid^{4\alpha} +\frac{\Omega^{2}\xi_3^{2}}{\mid\xi\mid^{2}}}\nu\mid\xi\mid^{2\alpha}I +
\frac{1}{\nu^{2}\mid\xi\mid^{4\alpha} +\frac{\Omega^{2}\xi_3^{2}}{\mid\xi\mid^{2}}}\frac{\Omega\xi_3}{\mid\xi\mid}
R(\xi)
\end{equation*}
we get

\begin{equation*}
\parallel y\parallel_{\dot{FB}^{4-2\alpha-\frac{3}{p}}_{p,q}}
\leq \frac{C}{\nu}\parallel F\parallel_{\dot{FB}^{4-4\alpha-\frac{3}{p}}_{p,q}}
\end{equation*} 
and, by Lemma \ref{productX}, we have 

\begin{equation*}
\parallel B(u,v)\parallel_{\dot{FB}^{4-2\alpha-\frac{3}{p}}_{p,q}}
\leq K
\parallel u\parallel_{\dot{FB}^{4-2\alpha-\frac{3}{p}}_{p,q}}
\parallel v\parallel_{\dot{FB}^{4-2\alpha-\frac{3}{p}}_{p,q}},
\end{equation*} 
where, in the context of the fixed point lemma (Lemma \ref{fixedpoint}), we identify 

\begin{equation*}
    y=\displaystyle{\int_{0}^{\infty}}S^{\alpha}_{\Omega}(\tau)\mathbb{P}F\,d\tau\,\,\mbox{and}\,\,
    B(u,v)=-\displaystyle{\int_{0}^{\infty}}S^{\alpha}_{\Omega}(\tau)\mathbb{P}\nabla(u\otimes v)\,d\tau.
\end{equation*}
Thus, there is $0<\varepsilon<\frac{1}{4K}$ such that, if 

\begin{equation*}
\parallel F\parallel_{\dot{FB}^{4-4\alpha-\frac{3}{p}}_{p,q}}<\frac{\varepsilon\nu}{C}
\end{equation*}
then $\parallel y\parallel_{\dot{FB}^{4-2\alpha-\frac{3}{p}}_{p,q}}\leq \varepsilon$. These estimatives permit us to conclude the proof of the Theorem \ref{teorema3}.
\end{proof}	
\begin{obs}
Let us define the space $X^{\alpha,\Omega}_{p,q}$ of singular functions associated to the fractional Navier-Stokes-Coriolis system, for  $1\leq p\leq +\infty$ and $1\leq q\leq +\infty$, as follow

\begin{equation}\label{General-space}
X^{\alpha,\Omega}_{p,q}=\{ f\in{\cal S}';
\parallel f\parallel_{X^{\alpha,\Omega}_{p,q}}=
\parallel\parallel \varphi_{j}w_1 \hat{f}
\parallel_{L^{p}}\parallel_{l^{q}(\mathbb{Z})}+
\parallel\parallel \varphi_{j} w_2 \hat{f}
\parallel_{L^{p}}\parallel_{l^{q}(\mathbb{Z})}<\infty
\},
\end{equation}
where

\begin{equation*}
w_1(\xi)=\frac{\nu\mid\xi\mid^{6-\frac{3}{p}}}{\nu^{2}\mid\xi\mid^{4\alpha+2}+\Omega^{2}\xi_3^{2}}I\,\,\mbox{and}\,\,
w_2(\xi)=\frac{\Omega\mid\xi_3\mid \mid\xi\mid^{5-2\alpha-\frac{3}{p}}}{
\nu^{2}\mid\xi\mid^{4\alpha+2}+\Omega^{2}\xi_3^{2}}R(\xi).
\end{equation*}

If we consider the estimate
$\parallel y\parallel_{\dot{FB}^{4-2\alpha-\frac{3}{p}}_{p,q}}\leq \parallel F\parallel_{X^{\alpha,\Omega}_{p,q}}$, then Theorem \ref{teorema3} it is also valid if we replace the condition 

\begin{equation*}
\parallel F\parallel_{\dot{FB}^{4-4\alpha-\frac{3}{p}}_{p,q}}<\frac{\varepsilon\nu}{C},
\end{equation*}
by the condition $\parallel F\parallel_{X^{\alpha,\Omega}_{p,q}}\leq \varepsilon$. This inequality allows to consider more general external forces. Moreover, in this case, it is valid the estimate

\begin{equation*}
\parallel u\parallel_{\dot{FB}^{4-2\alpha-\frac{3}{p}}_{p,q}} 
\leq D\parallel F\parallel_{X^{\alpha,\Omega}_{p,q}},
\end{equation*}
for some positive constant $D$. 
\end{obs}
Next lemma describes the asympotic behaviour of non-stationary solutions as time is large enough.
\begin{coro}
Let $u_\infty$ be the stationary mild solution obtained from Theorem \ref{teorema3} associated with the external force $F=F(x)$. Let us consider $v_0\in \dot{FB}^{4-2\alpha-\frac{3}{p}}_{p,q}$ and $G\in{\cal L}^{\infty}(I;\dot{FB}^{4-4\alpha-\frac{3}{p}}_{p,q})$	such that

\begin{equation*}
\parallel v_0\parallel_{\dot{FB}^{4-2\alpha-\frac{3}{p}}_{p,q}}+
\parallel G\parallel_{{\cal L}^{\infty}(I;\dot{FB}^{4-4\alpha-\frac{3}{p}}_{p,q})}\leq \varepsilon<\frac{1}{4K}.
\end{equation*}
\end{coro}
If we assume that

\begin{equation*}
\displaystyle{\lim_{t\rightarrow \infty}}
\parallel S^{\alpha}_{\Omega}(t)(v_0-u_\infty)\parallel_{\dot{FB}^{4-2\alpha-\frac{3}{p}}_{p,q}}=0 \,\,\mbox{and}\,\,
\displaystyle{\lim_{t\rightarrow \infty}}
\parallel G(t)-F\parallel_{\dot{FB}^{4-4\alpha-\frac{3}{p}}_{p,q}}=0
\end{equation*}
then the solution $v=v(x,t)$ of the (FNSC)-system associated to the initial data $v_0$ and the external force $G$ converges to the stationary solution $u_\infty$, that is, 

\begin{equation*}
\displaystyle{\lim_{t\rightarrow \infty}}
\parallel v(t)-u_\infty\parallel_{\dot{FB}^{4-2\alpha-\frac{3}{p}}_{p,q}}=0
\end{equation*}
\begin{proof}
Let us consider the external force $F$ and the stationary mild solution $u_\infty$ as a constant time-dependent function, that is,

\begin{equation*}
u_\infty(t)=u_\infty\,\,\mbox{for all}\,\,t>0\,\,\mbox{and}\,\,
F(t)=F\,\,\mbox{for all}\,\,t>0.
\end{equation*}	
By Lemma \ref{equivStationary}, we have

\begin{equation*}
u_\infty=S^{\alpha}_{\Omega}(t)u_\infty-
\displaystyle{\int_{0}^{t}}
S^{\alpha}_{\Omega}(t-\tau)\mathbb{P}\nabla\cdot(u_\infty\otimes u_\infty)\,d\tau+
\displaystyle{\int_{0}^{t}}
S^{\alpha}_{\Omega}(\tau)\mathbb{P}F\,d\tau.
\end{equation*}
Since $u_\infty\in \dot{FB}^{4-2\alpha-\frac{3}{p}}_{p,q}$ and $v_0\in \dot{FB}^{4-2\alpha-\frac{3}{p}}_{p,q}$ and external forces $F(t)=F\in{\cal L}^{\infty}(I;\dot{FB}^{4-4\alpha-\frac{3}{p}}_{p,q})$ and $G\in {\cal L}^{\infty}(I;\dot{FB}^{4-4\alpha-\frac{3}{p}}_{p,q})$, Theorem \ref{teorema2} gives 

\begin{equation*}
\displaystyle{\lim_{t\rightarrow \infty}}
\parallel v(t,\cdot)-u_\infty(t,\cdot)\parallel_{\dot{FB}^{4-2\alpha-\frac{3}{p}}_{p,q}}=0.
\end{equation*}
This is the required conclusion.
\end{proof}
In the last part of this work, we shall proof Theorem \ref{teorema4}. Here we denote $X^{\alpha,\Omega}_{p}=X^{\alpha,\Omega}_{p,p}$ and we study the relation between the external force with the Coriolis parameter for the stationary fractional Navier-Stokes-Coriolis system
 
\begin{equation}\label{eq2.8}
\nu(-\Delta)^{\alpha}u+\Omega e_3\times u+(u\cdot\nabla)u+\nabla p=F,\,\,\mbox{div}\,u=0\,\,\mbox{in}\,\,\mathbb{R}^{3}.
\end{equation}
In order to prove Theorem \ref{teorema4}, we just need to show that, for each 
$F\in \dot{FB}^{4-4\alpha-\frac{3}{p}}_{p,p}$ and each $\varepsilon$ there is 
$\Omega_0$ such that for all $\mid\Omega\mid\geq \Omega_0$, we have

\begin{equation}\label{HH0}
\parallel F\parallel_{X^{\alpha, \Omega}_{p}} 
\leq \varepsilon.
\end{equation}

\begin{obs}
Let us observe that $\dot{FB}^{4-4\alpha-\frac{3}{p}}_{p,p}\subset X^{\alpha, \Omega}_{p}$. In fact, to show this inclusion it is enough to observe that

\begin{equation}\label{eq4.9}
\frac{\nu\mid\xi\mid^{2\alpha+2}}{\nu^{2}\mid\xi\mid^{4\alpha+2}+\Omega^{2}
\mid\xi_3\mid^{2}}=
\displaystyle{\int_{0}^{\infty}}
e^{-\nu t\mid\xi\mid^{2\alpha}}
\cos(\frac{\Omega \xi_3}{\mid\xi\mid}t)\,dt
\leq
\displaystyle{\int_{0}^{\infty}}
e^{-\nu t\mid\xi\mid^{2\alpha}}\,dt=\frac{\mid\xi\mid^{-2\alpha}}{\nu}
\end{equation}
and 

\begin{equation}\label{eq4.10}
\frac{\Omega\xi_3\mid\xi\mid}{\nu^{2}\mid\xi\mid^{4\alpha+2}+\Omega^{2}
	\mid\xi_3\mid^{2}}=
\displaystyle{\int_{0}^{\infty}}
e^{-\nu t\mid\xi\mid^{2\alpha}}
\sin(\frac{\Omega \xi_3}{\mid\xi\mid}t)\,dt
\leq
\displaystyle{\int_{0}^{\infty}}
e^{-\nu t\mid\xi\mid^{2\alpha}}\,dt=\frac{\mid\xi\mid^{-2\alpha}}{\nu},
\end{equation}
since inequalities $(\ref{eq4.9})$ and $(\ref{eq4.10})$ give the following estimates

\begin{equation*}
\parallel w_{1}\hat{F}\parallel_{L^{p}}\leq \frac{1}{\nu}\parallel F\parallel_{\dot{FB}^{4-4\alpha-\frac{3}{p}}_{p,p}}
\end{equation*}
and 

\begin{equation*}
\parallel w_{2}\hat{F}\parallel_{L^{p}}\leq \frac{1}{\nu}\parallel F\parallel_{\dot{FB}^{4-4\alpha-\frac{3}{p}}_{p,p}},
\end{equation*}
for every $F\in \dot{FB}^{4-4\alpha-\frac{3}{p}}_{p,p}$
\end{obs}
Now, we proceed as P. Konieczny and T. Yoneda (\cite{KoniecznyYonedaDispersiveStationaryNSC}, 2011). We decompose $\mathbb{R}^{3}=
A_\delta \cup B_\delta\cup C_\delta$, where $A_\delta=\{\xi;\mid\xi_{3}\mid>\delta\,\,\mbox{and}\,\,\delta<\mid\xi\mid\leq\frac{1}{\delta}\}$, $B_\delta=\{\xi; \mid\xi_{3}\mid>\delta\,\,\mbox{and}\,\, \mid\xi\mid>\frac{1}{\delta}\}$ and $C_\delta=\{ \xi;\mid\xi_{3}\mid\leq\delta\}$.

Since $F\in\dot{FB}^{4-4\alpha-\frac{3}{p}}_{p,p}$ is given, there is a compact $K\subset \mathbb{R}^{3}$ such that

\begin{equation}\label{HH1}
\parallel \mid\xi\mid^{4-4\alpha-\frac{3}{p}}F
\parallel_{L^{p}(\mathbb{R}^{3}-K)}\leq \frac{\varepsilon}{6}\cdot\nu.
\end{equation}
Observe that for $\delta>$ small enough, the set $K\cap B_{\delta}$ is empty and $\mid K\cap C_{\delta}\mid\leq C\delta^{3}$, for some constant $C>0$. Therefore $\displaystyle{\lim_{\delta\rightarrow 0}} \mid K \cap (B_\delta \cup C_\delta)\mid=0$. Thus, we can choose $0<\delta$ small enough such that

\begin{equation}\label{HH2}
\parallel F\parallel_{X^{\alpha, \Omega}_{p}(K\cap (B_\delta \cup C_\delta))}\leq \frac{\varepsilon}{3}.
\end{equation}
With this fixed $\delta$ we can estimate the integral over $K\cap A_\delta$
 as follows,
 
 \begin{equation*}
 \left[
 \displaystyle{\int_{K\cap A_\delta}}
 \left(
 \frac{\nu\mid\xi\mid^{6-\frac{3}{p}}}{\nu^{2}\mid\xi\mid^{4\alpha+2} +\Omega^{2}\xi_3^{2}}\mid\hat{F}\mid
 \right)^{p}\,d\xi
 \right]^{\frac{1}{p}}
 \leq
 \frac{\nu(1/\delta)^{4\alpha+2}}{\nu^{2}\delta^{4\alpha+2}+\Omega^{2}\delta^{2}}\parallel F\parallel_{\dot{FB}^{4-4\alpha-\frac{3}{p}}_{p,p}}
 \frac{1}{\nu}
 \end{equation*}
and thus we get an $ \Omega_0=\Omega_0(\varepsilon,\delta,\parallel F\parallel_{\dot{FB}^{4-4\alpha-\frac{3}{p}}_{p,p}
})$ such that for all $\mid \Omega\mid\geq \Omega_0$ we have

\begin{equation}\label{HH3}
\left[
\displaystyle{\int_{K\cap A_\delta}}
\left(
\frac{\nu\mid\xi\mid^{6-\frac{3}{p}}}{\nu^{2}\mid\xi\mid^{4\alpha+2} +\Omega^{2}\xi_3^{2}}\mid\hat{F}\mid
\right)^{p}\,d\xi
\right]^{\frac{1}{p}}
\leq
\frac{\varepsilon}{6}.
\end{equation}
Also, we have 

\begin{equation*}
\left[
\displaystyle{\int_{K\cap A_\delta}}
\left(
\frac{\Omega\mid\xi_3\mid\mid\xi\mid^{5-2\alpha-\frac{3}{p}}}{\nu^{2}\mid\xi\mid^{4\alpha+2} +\Omega^{2}\xi_3^{2}}\mid R(\xi)\mid\mid\hat{F}\mid
\right)^{p}\,d\xi
\right]^{\frac{1}{p}}
\leq
\frac{\nu(1/\delta)^{2\alpha+2}}{\nu^{2}\delta^{4\alpha+2}+\Omega^{2}\delta^{2}}\parallel F\parallel_{\dot{FB}^{4-4\alpha-\frac{3}{p}}_{p,p}}
\end{equation*}
and therefore we have the estimate

\begin{equation}\label{HH4}
\left[
\displaystyle{\int_{K\cap A_\delta}}
\left(
\frac{\Omega\mid\xi_3\mid\mid\xi\mid^{5-2\alpha-\frac{3}{p}}}{\nu^{2}\mid\xi\mid^{4\alpha+2} +\Omega^{2}\xi_3^{2}}\mid R(\xi)\mid\mid\hat{F}\mid
\right)^{p}\,d\xi
\right]^{\frac{1}{p}}
\leq\frac{\varepsilon}{6},
\end{equation}
for $\mid \Omega\mid\geq \Omega_0$.
Therefore estimatives $(\ref{HH1})$, $(\ref{HH2})$, 
$(\ref{HH3})$ and $(\ref{HH4})$ give the desired estimative $(\ref{HH0})$.
In order to cover the case for $F\in \dot{FB}^{4-4\alpha-\frac{3}{p}}_{p,\infty}$ we proceed as follow: As in the above estimate 
we just need to prove that, if a constant $\varepsilon>0$ is given, $F\in \dot{FB}^{4-4\alpha-\frac{3}{p}}_{
p,\infty}$ is given and there is $J\in\mathbb{N}$ such that

\begin{equation}\label{condicionF}
\displaystyle{\sup_{\mid j\mid>J}}\left(
\parallel \varphi_{j}w_1\hat{F}\parallel_{L^{p}} + 
\parallel\varphi_{j}w_2\hat{F}\parallel_{L^{p}}\right)
\end{equation}
is small enough (for instance less than $\frac{\varepsilon}{3})$, there exist $\Omega_0>0$ such that for all $\mid\Omega\mid\geq\Omega_0$ we have

\begin{equation}\label{asterisco}
\parallel F\parallel_{X^{\alpha,\Omega}_{p,\infty}}\leq 
\varepsilon
\end{equation}
In fact, given such function $F$, for each $j\in\mathbb{Z}$ such that $\mid j\mid\leq J$, we have a compact set $K_j$ such that 

\begin{equation*}
\parallel \varphi_{j}\mid\xi\mid^{4-4\alpha-\frac{3}{p}}\hat{F}
\parallel_{L^{p}(\mathbb{R}^{3}-K_j)}\leq \frac{\varepsilon}{3}
\end{equation*}
Let us denote $K=K_{-J}\cup \dots \cup K_{-1}\cup K_1\cup \dots \cup K_J$ which is also compact.
Thus, for this compact set $K$ the following estimate is true

\begin{equation}\label{I1}
\parallel \varphi_{j}\mid\xi\mid^{4-4\alpha-\frac{3}{p}}\hat{F}
\parallel_{L^{p}(\mathbb{R}^{3}-K)}\leq\frac{\varepsilon}{3},
\end{equation}
for all $\mid j\mid\leq J$. Since $\displaystyle{\lim_{\delta\rightarrow 0}}\mid K\cap (B_\delta \cup C_\delta)\mid=0$ we can choose $0<\delta$ small enough such that

\begin{equation*}
\parallel w_1\varphi_j\hat{F}\parallel_{L^{p}(K\cap (B_\delta\cup C_\delta))}
\leq \frac{\varepsilon}{12}\,\,\mbox\,\,
\parallel w_2\varphi_j\hat{F}\parallel_{L^{p}(K\cap (B_\delta\cup C_\delta)) }
\leq \frac{\varepsilon}{12},
\end{equation*} 
for all $\mid j\mid\leq J$.
For this choosen value of $\delta$ independent of $j$ and since

\begin{equation*}
\parallel w_1\varphi_j\hat{F}\parallel_{L^{p}(K\cap A_{\delta} )}
 \leq
 \frac{\nu(1/\delta)^{4\alpha+2}}{\nu^{2}\delta^{4\alpha+2}+\Omega^{2}\delta^{2}}\parallel F\parallel_{\dot{FB}^{4-4\alpha-\frac{3}{p}}_{p,\infty}}
 \frac{1}{\nu},
 \end{equation*}
we get an $ \Omega_0=\Omega_0(\varepsilon,\delta,\parallel F\parallel_{\dot{FB}^{4-4\alpha-\frac{3}{p}}_{p,\infty}
})$ such that for all $\mid \Omega\mid\geq \Omega_0$ we have

\begin{equation*}
\parallel w_1\varphi_j\hat{F}\parallel_{L^{p}(K\cap A_{\delta} )}
 \leq \frac{\varepsilon}{12}.
\end{equation*}
Similarly we obtain

\begin{equation*}
\parallel w_2\varphi_j\hat{F}\parallel_{L^{p}(K\cap A_{\delta} )}
 \leq \frac{\varepsilon}{12},
\end{equation*}
for each $\mid j\mid\leq J$ and for all $\mid \Omega\mid\geq \Omega_0$.
Finally, assumption $(\ref{condicionF})$ on $F$ gives the desired estimate $(\ref{asterisco})$. 
\begin{obs}
Since we obtain stationary solutions for the (FNSC)-system, it is possible to consider this approach as a Fourier tranform-based approach for a elliptic system. To see a recent harmonic analysis approach for a singular elliptic equation we suggest (\cite{Castaneda-C-Ferreira-Singular-Elliptic-2021}, 2021), where the authors studied a class of nonlinear elliptic boundary value problem in the half-space $\mathbb{R}^{n}_{+}$ by considering an integral representation based on Fourier-transform in the first $n-1$ variables and dealing with the last variable $x_n$ as a time-variable. They establish an existence and uniqueness framework for the equation in ${\cal PM}^{a}$-spaces and study some regularity properties.
\end{obs}

\textbf{Acknowledgment}:\newline
This work was supported by Cnpq (Brazil) and by University of Campinas, SP, Brazil.

\end{document}